\documentclass{amsart}

\usepackage{amssymb}
\usepackage{amsmath}
\usepackage{lscape}
\usepackage{enumerate}
\usepackage{rotating}

\newcommand{\erase}[1]{}

\newtheorem{theorem}{Theorem}[section]
\newtheorem{lemma}[theorem]{Lemma}
\newtheorem{proposition}[theorem]{Proposition}
\newtheorem{corollary}[theorem]{Corollary}

\newtheorem{_algorithm}[theorem]{Algorithm}
\newenvironment{algorithm}{\begin{_algorithm}\rm}{\hfill \rule{3pt}{6pt}\end{_algorithm}}

\newtheorem{_procedure}[theorem]{Procedure}

\newtheorem{_definition}[theorem]{Definition}
\newenvironment{definition}{\begin{_definition}\rm}{\end{_definition}}

\newtheorem{_remark}[theorem]{\it Remark}
\newenvironment{remark}{\begin{_remark}\rm}{\end{_remark}}

\newtheorem{_example}[theorem]{Example}

\newtheorem{_assumption}[theorem]{Assumption}

\newtheorem{_construction}[theorem]{Construction}

\newtheorem{_claim}[theorem]{Claim}

\newtheorem{_conjecture}[theorem]{Conjecture}

\numberwithin{equation}{section}
\numberwithin{table}{section}
\numberwithin{figure}{section}
\renewcommand{\qed}{\hfill {$\Box$}}

\newcommand{\C}{\mathord{\mathbb C}}

\newcommand{\F}{\mathord{\mathbb F}}

\renewcommand{\P}{\mathord{\mathbb  P}}
\newcommand{\Q}{\mathord{\mathbb  Q}}
\newcommand{\R}{\mathord{\mathbb R}}

\newcommand{\Z}{\mathord{\mathbb Z}}

\newcommand{\FFF}{\mathord{\mathcal F}}
\newcommand{\GGG}{\mathord{\mathcal G}}
\newcommand{\HHH}{\mathord{\mathcal H}}

\newcommand{\PPP}{\mathord{\mathcal P}}

\newcommand{\RRR}{\mathord{\mathcal R}}

\newcommand{\TTT}{\mathord{\mathcal T}}

\newcommand{\VVV}{\mathord{\mathcal V}}
\newcommand{\WWW}{\mathord{\mathcal W}}

\newcommand{\SSSS}{\mathord{\mathfrak S}}

\newcommand{\maprightsp}[1]{\; \smash{\mathop{\; \longrightarrow \; }\limits\sp{#1}}\; }

\newcommand{\maprightinj}{\; \smash{\mathop{\; \inj \; }}\; }

\newcommand{\mapdown}{\phantom{\Big\downarrow}\hskip -8pt \downarrow}
\newcommand{\mapdownright}[1]{\mapdown\rlap{$\vcenter{\hbox{$\scriptstyle#1$}}$}}

\newcommand{\mapdownsurj}{
\hbox{$\bigm\downarrow$}
\llap{\hbox{\raise 2pt\hbox{$\bigm\downarrow$}}}%
\vstrechmapdown
}

\newcommand{\mapupsurj}{
\hbox{$\bigm\uparrow$}
\llap{\hbox{\raise 2pt\hbox{$\bigm\uparrow$}}}%
\vstrechmapup
}

\newcommand{\inj}{\hookrightarrow}
\newcommand{\surj}{\mathbin{\to \hskip -7pt \to}}

\newcommand{\isom}{\xrightarrow{\sim}}

\newcommand{\set}[2]{\{\; {#1} \; \mid \; {#2} \;  \}}
\newcommand{\shortset}[2]{\{ {#1} \,|\, {#2}   \}}

\newcommand{\sethd}[3]{\left\{\;\;  {#1}\;\; \left|\;\;  \vcenter{\hbox{\parbox{#2}{#3}}}\;\;  \right. \right\}}

\newcommand{\gen}[1]{\langle {#1}  \rangle}

\newcommand{\tensor}{\otimes}

\newcommand{\sprime}{\sp\prime}

\newcommand{\spprime}{\sp{\prime\prime}}

\newcommand{\sperp}{\sp{\perp}}

\newcommand{\dual}{\sp{\vee}}

\newcommand{\inv}{\sp{-1}}

\newcommand{\Hom}{\mathord{\mathrm{Hom}}}

\newcommand{\OG}{\mathord{\mathrm{O}}}

\newcommand{\id}{\mathord{\mathrm{id}}}

\newcommand{\Ker}{\operatorname{\mathrm{Ker}}\nolimits}

\newcommand{\Aut}{\operatorname{\mathrm{Aut}}\nolimits}

\newcommand{\pr}{\mathord{\mathrm{pr}}}

\newcommand{\rank}{\operatorname{\mathrm{rank}}\nolimits}

\newcommand{\closure}[1]{\overline{#1}}

\newcommand{\mystruth}[1]{\phantom{\hbox{\vrule height #1}}}
\newcommand{\mystrutd}[1]{\phantom{\hbox{\vrule depth #1}}}

\newcommand{\barX}{\overline{X}}
\newcommand{\barY}{\overline{Y}}
\newcommand{\barD}{\overline{D}}
\newcommand{\barC}{\overline{C}}
\newcommand{\barPPP}{\overline{\PPP}}

\newcommand{\spstar}{^{*}}

\newcommand{\enrinvol}{\varepsilon}
\newcommand{\intform}[1]{\langle #1\rangle}
\newcommand{\intf}[1]{\langle #1\rangle}
\newcommand{\intfvoid}{\intf{\phantom{\cdot}, \phantom{\cdot}}}
\newcommand{\intfX}[1]{\langle #1\rangle_{X}}
\newcommand{\intfY}[1]{\langle #1\rangle_{Y}}
\newcommand{\DX}{D_X}
\newcommand{\DY}{D_Y}
\newcommand{\hX}{h_X}
\newcommand{\hY}{h_Y}
\newcommand{\SX}{S_X}
\newcommand{\SY}{S_Y}
\newcommand{\SZ}{S_Z}
\newcommand{\aut}{\mathord{\mathrm{aut}}}
\newcommand{\LR}{L_{\R}}
\newcommand{\LQ}{L_{\Q}}
\newcommand{\prim}{\mathord{\rm{prim}}}
\newcommand{\tilv}{\tilde{v}}

\newcommand{\Rhalf}{R_{\mathord{\rm half}}}
\newcommand{\Rfull}{R_{\mathord{\rm full}}}

\newcommand{\RDP}{\mathord{\rm RDP}}

\newcommand{\Lten}{L_{10}}
\newcommand{\Zgen}{Z_{\mathord{\rm gen}}}
\newcommand{\SZgen}{S_{\Zgen}}

\newcommand{\nefbig}{N}

\newcommand{\VR}{V_{\R}}
\newcommand{\PPPV}{\PPP_V}

\newcommand{\barv}{\bar{v}}
\newcommand{\baru}{\bar{u}}
\newcommand{\barg}{\bar{g}}
\newcommand{\barsigma}{\bar{\sigma}}

\newcommand{\simw}{\mathbin{\,\mathop{\sim}\limits^{a}\,}}
\newcommand{\simwvoid}{\mathbin{\mathop{\sim}\limits^{a}}}

\newcommand{\vare}{\varepsilon}

\newcommand{\algoG}{\mathord{\tt G}}
\newcommand{\algof}{\mathord{\tt f}}

\newcommand{\algoid}{\id}

\newcommand{\maxSRCdeg}{46}

\newcommand{\wtautY}{\aut\sprime (Y)}

\newcommand{\YVI}{Y_{\rm{VI}}}

\newcommand{\HS}{H_S}
\newcommand{\dcolon}{\;\colon\,}

\newcommand{\GO}{\mathord{\rm GO}}

\begin{document}
\title[Enriques surface associated with a  quartic Hessian surface]%
{On an Enriques surface associated with a  quartic Hessian surface}
\author{Ichiro Shimada}
\address{Department of Mathematics,
Graduate School of Science,
Hiroshima University,
1-3-1 Kagamiyama,
Higashi-Hiroshima,
739-8526 JAPAN}
\email{ichiro-shimada@hiroshima-u.ac.jp}
\thanks{This work was supported by JSPS KAKENHI Grant Number 16H03926 and 16K13749.}
\dedicatory{Dedicated to Professor Jonghae Keum on the occasion of his 60th birthday}

\begin{abstract}
Let $Y$ be  a complex Enriques surface
whose  universal cover $X$ is birational to a general  quartic Hessian surface.
Using the result on the automorphism group of $X$
due to Dolgachev and Keum,
we obtain
a finite presentation of the automorphism group of  $Y$.
The list of elliptic fibrations on $Y$
and the list of combinations of rational double points that can appear on a surface birational to $Y$
are presented.
As an application,
a set of generators of
the automorphism group of the generic Enriques surface is calculated explicitly.
\end{abstract}


\maketitle
\section{Introduction}
We work over the complex number field $\C$.
An involution on a $K3$ surface is called an \emph{Enriques involution}
if it has no fixed-points.
Let $\barX$ be a general quartic Hessian surface, which
means that $\barX$  is the quartic surface in $\P^3$ defined by the equation
$$
\det\left( \frac{\partial^2 F}{\partial x_i\partial x_j}\right)=0,
$$
where $F=F(x_1, \dots, x_4)$ is a general cubic homogeneous polynomial.
Then $\barX$ has ten ordinary nodes $p_{\alpha}$ as its only singularities,
and contains exactly ten lines $\ell_{\beta}$.
Let $A$ denote the set of subsets $\alpha$ of $\{1, 2, 3, 4, 5\}$
with $|\alpha|=3$,
and $B$  the set of subsets $\beta$ of $\{1, \dots, 5\}$
with $|\beta|=2$.
Then  $p_{\alpha}$ and  $\ell_{\beta}$
can be indexed by $\alpha\in A$ and $\beta\in B$, respectively,
in such a way that
$p_{\alpha}\in \ell_{\beta}$ if and only if $\alpha \supset \beta$.
Let $X\to \barX$ be the minimal resolution,
let $E_{\alpha}$ be the exceptional curve over $p_{\alpha}$,
and let $L_{\beta}$ be the strict transform of  $\ell_{\beta}$.
It is classically known (see Dolgachev and Keum~\cite{MR1897389}) that  the $K3$ surface $X$ has
an Enriques involution $\enrinvol$
that interchanges  $E_\alpha$ and
$L_{\bar{\alpha}}$ for each $\alpha\in A$,
where $\bar{\alpha}:=\{1, \dots, 5\}\setminus \alpha$.
We denote the  quotient morphism by
$$
\pi\;\colon \;X\; \to \;Y:=X/\gen{\enrinvol}.
$$
\par
The first application of Borcherds method~(\cite{MR913200},~\cite{MR1654763})
to the automorphism group of $K3$ surfaces
was given by Kondo~\cite{MR1618132}.
A set of generators of
the automorphism group $\Aut(X)$ of the $K3$ surface $X$ above was obtained in Dolgachev and Keum~\cite{MR1897389}
by this method.
On the other hand, we presented in~\cite{MR3456710}
a computer algorithm for
Borcherds method.
Using this computational tool,
we obtain an explicit description of $\Aut(X)$ and a fundamental domain $\DX $ of its action on the cone
$$
\nefbig(X):=\set{x\in \PPP_X}{\;\intf{x, [C]}\ge 0\;\;\textrm{for all curves $C$ on $X$}\;}
$$
in $ \SX\tensor \R$,
where $\SX$ is  the N\'eron-Severi lattice of $X$
with the intersection form $\intfvoid$,
$\PPP_X$ is the connected component of $\shortset{x\in \SX\tensor\R}{\intf{x,x}>0}$ containing an ample class,
and $[C]\in \SX$ is the class of a curve $C\subset X$.
\par
By analyzing  this result,
we obtain the following results on the automorphism group $\Aut(Y)$
of the Enriques surface $Y$.
%
%
%
Let $\iota_{\alpha}\colon X\to X$ denote the involution
of $X$ induced by the double covering   $\barX\to \P^2$ obtained from
the projection with the center $p_{\alpha}\in\barX$.
In~\cite{MR1897389},
it was proved that  $\iota_{\alpha}$ commutes with  $\enrinvol$.
Hence $\iota_{\alpha}$ induces an involution $j_{\alpha}\colon Y\to Y$ of $Y$.
\begin{theorem}\label{thm:genAut}
The automorphism group $\Aut(Y)$ of  $Y$ is generated by the ten involutions $j_{\alpha}$.
The following relations form a set of defining relations of $\Aut(Y)$ with respect to these generators $j_{\alpha}$;
$$
j_{\alpha}^2=\id
$$
for each ordinary node $p_{\alpha}$,
$$
(j_{\alpha} j_{\alpha\sprime} j_{\alpha\spprime})^2=\id
$$
for each triple $(p_\alpha, p_{\alpha\sprime}, p_{\alpha\spprime})$  of distinct three ordinary nodes
such that  there exists a line in $\barX$ passing  through
 $p_{\alpha},  p_{\alpha\sprime}, p_{\alpha\spprime}$, and
$$
(j_{\alpha} j_{\alpha\sprime} )^2=\id
$$
 for each  pair $(p_\alpha, p_{\alpha\sprime})$ of distinct  ordinary nodes
 such that the line in $\P^3$ passing through  $p_{\alpha}$ and $p_{\alpha\sprime}$ is not contained in $\barX$.
\end{theorem}
\begin{remark}\label{rem:Dolgachev}
Recently Dolgachev~\cite{DolgachevArxiv2016} studied the group generated
by the involutions  $j_{\alpha}$ (that is, $\Aut(Y)$ by Theorem~\ref{thm:genAut}), and
showed that this group is isomorphic to a subgroup of
$\Gamma\rtimes\SSSS_5$,
where $\Gamma$ is a group
isomorphic to the Coxeter group
with the anti-Petersen graph as its Coxeter graph.
See~Corollary 4.4~of~\cite{DolgachevArxiv2016}.
(This corollary was also known to Mukai.)
\end{remark}
\begin{remark}\label{rem:MukaiOhashi}
Mukai and Ohashi informed us that they also proved,
without computer-aided calculations,
that $\Aut(Y)$ is generated by $j_{\alpha}$ ($\alpha\in A$).
See also~\cite{MukaiOhashi}.
\end{remark}
Let $\SY$ denote the  lattice of numerical equivalence classes of divisors on $Y$,
which is isomorphic to $H^2(Y, \Z)/(\text{torsion})$
equipped with the cup-product.
By the result of~\cite{MR759266}, ~\cite{KondoFinite}  and~\cite{MR2740697},
we know that the action of  $\Aut(Y)$  on  $\SY$  is faithful.
Theorem~\ref{thm:genAut}
is proved by investigating this faithful action.
More precisely,
let $\PPP_Y$ denote the connected component
of $\shortset{y\in \SY\tensor \R}{\intf{y,y}> 0}$
containing an ample class.
We put
$$
\nefbig(Y):=\set{y\in \PPP_Y}{\intf{y, [C]}\ge 0\;\;\textrm{for all curves $C$ on $Y$}\;}.
$$
It is obvious that $\Aut(Y)$ acts on  $\nefbig(Y)$.
We give a description of  a fundamental domain $\DY$
of the action of  $\Aut(Y)$ on
$\nefbig(Y)$.
For $v\in \SY\tensor\R$ with $\intf{v, v}<0$,
let $(v)\sperp$ denote the hyperplane
in $\PPP_Y$  defined by  $\intf{v, x}=0$.
\begin{theorem}\label{thm:DY}
There exists a fundamental domain $\DY$  of the action of $\Aut(Y)$ on $\nefbig(Y)$ with the following properties.
\begin{enumerate}[{\rm (1)}]
\item
The fundamental domain $\DY$ is bounded by $10+10$ hyperplanes
$(\baru_{\alpha})\sperp$ and $(\barv_{\alpha})\sperp$,
where $\alpha$ runs through the set $A$.
\item
For each $\alpha\in A$,
the vector $\baru_{\alpha}$ is the class of
the  smooth rational curve $\pi(E_{\alpha})=\pi(L_{\bar{\alpha}})$ on $Y$,
and hence
 $(\baru_{\alpha})\sperp$ is a  hyperplane bounding  $\nefbig(Y)$.
\item
For each $\alpha\in A$,
the involution $j_{\alpha}\in \Aut(Y)$ maps $\DY$ to the chamber adjacent to $\DY$
across the wall $\DY \cap (\barv_{\alpha})\sperp$ of $\DY$.
\end{enumerate}
\end{theorem}
Let $Z$ be an Enriques surface.
Then an elliptic fibration $\phi\colon Z\to \P^1$
has exactly two multiple fibers $2E_1$ and $2E_2$.
\begin{theorem}\label{thm:ellfibs}
Up to the action of $\Aut(Y)$,
the Enriques surface $Y$ has exactly $1+10+5+5$ elliptic fibrations.
Their $ADE$-types of reducible fibers are given in Table~\ref{table:ellfibs},
in which $\Rfull$ and $\Rhalf$ denote
the $ADE$-types of non-multiple reducible fibers
and  of the half of the multiple reducible fibers,
respectively.
\end{theorem}
An \emph{$\RDP$-configuration} on an Enriques surface $Z$
is the exceptional divisor of a birational morphism
$Z\to \bar{Z}$,
where $\bar{Z}$ has only rational double points as its singularities.
The support of an  $\RDP$-configuration is an $ADE$-configuration of smooth rational curves.
\begin{table}
$$
\arraycolsep=6pt
\renewcommand{\arraystretch}{1.1}
\begin{array}{cc}
[\Rfull, \Rhalf] & \textrm{number} \\
\hline
{[\emptyset, A_4]} & 1 \\
{[A_5+A_1, \emptyset]} & 10 \\
{[D_5, \emptyset]} & 5 \\
{[E_6, \emptyset]} & 5 \\
\end{array}
$$
\vskip .2cm
\caption{Elliptic fibrations on $Y$}\label{table:ellfibs}
\end{table}
\begin{theorem}\label{thm:RDP}
Up to the action of  $\Aut(Y)$,
the Enriques surface $Y$ has exactly $33$ non-empty $\RDP$-configurations.
Their $ADE$-types  are given in Table~\ref{table:RDP}.
\footnote{Added in 2020 August:
In the previous version of this paper, the numbers of $\Aut(Y)$-equivalence classes
of rational double points
were calculated wrongly.
See Section~\ref{subsec:RDP}.}
\end{theorem}
\begin{table}
\renewcommand{\arraystretch}{1.1}
$$
\arraycolsep=6pt
\begin{array}{cc}
 \hbox{\rm $ADE$-type} & {\rm number} \\
 \hline
E_{6} & 1\\
A_{5}+A_{1} & 5\\
3A_{2} & 1\\
D_{5} & 1\\
A_{5} & 1\\
A_{4}+A_{1} & 1\\
A_{3}+2A_{1} & 5\\
2A_{2}+A_{1} & 1\\
D_{4} & 1\\
A_{4} & 1\\
\end{array}
\qquad\quad
\begin{array}{cc}
 \hbox{\rm $ADE$-type} & {\rm number} \\
 \hline
A_{3}+A_{1} & 1\\
2A_{2} & 1\\
A_{2}+2A_{1} & 1\\
4A_{1} & 5\\
A_{3} & 1\\
A_{2}+A_{1} & 1\\
3A_{1} & 2\\
A_{2} & 1\\
2A_{1} & 1\\
A_{1} & 1\\
\end{array}
$$
\vskip .2cm
\caption{$\RDP$-configurations on $Y$}\label{table:RDP}
\end{table}
\begin{remark}
In~\cite{ShimadaRDPEnr},
all $\RDP$-configurations on complex Enriques surfaces are classified
by some lattice theoretic equivalence relation.
\end{remark}
The  lattice $\SZ$ of numerical equivalence classes of divisors on an Enriques surface $Z$ is isomorphic to
the even unimodular hyperbolic lattice $\Lten$ of rank $10$,
which is unique up to isomorphism.
The group $\OG^+(\Lten)$
of isometries of $\Lten$ preserving a positive cone $\PPP_{10}$ of $\Lten\tensor\R$
is generated by the reflections with respect to the roots.
Vinberg~\cite{MR0422505} determined the shape of
a standard fundamental domain of
the action of $\OG^+(\Lten)$ on $\PPP_{10}$.
(See Section~\ref{subsec:L10}.)
Hence we call these fundamental domains
\emph{Vinberg chambers}.
\par
In~\cite{MR718937}, Barth and Peters determined the automorphism group $\Aut(\Zgen)$ of
a \emph{generic} Enriques surface $\Zgen$. (See also Nikulin~\cite{NikulinEnriques}.)
Let $\PPP_{\Zgen}$ be the positive cone of $\SZgen\tensor\R$
containing an ample class.
We identify  $\SZgen$  with  $\Lten$
by an isometry  that maps $\PPP_{\Zgen}$ to $\PPP_{10}$.
Since $\Zgen$  contains no smooth rational curves,
we have
$\nefbig(\Zgen)= \PPP_{10}$.
Note that the discriminant form $q_{\Lten(2)}$  of $\Lten(2)$
is a quadratic form over $\F_2$ with Witt defect $0$,
and hence its automorphism group $\OG(q_{\Lten(2)})$ is isomorphic to $\GO^+_{10}(2)$
in the notation of~\cite{ATLAS}.
Moreover, the natural homomorphism
$$
\rho\colon \OG^+(\Lten) \to \OG(q_{\Lten(2)})\cong \GO^+_{10}(2)
$$
is surjective.
It was shown in~\cite{MR718937} that
the natural representation of $\Aut(\Zgen)$
on $\SZgen\cong \Lten$ identifies $\Aut(\Zgen)$
with the kernel of $\rho$.
In particular,
$\Aut(\Zgen)$ is isomorphic to a normal  subgroup of $\OG^+(\Lten)$ with index
$$
|\GO^+_{10}(2)|=2^{21}\cdot 3^5 \cdot 5^2\cdot 7\cdot 17\cdot 31=46998591897600.
$$
By the following theorem,
we can describe
the way how the automorphism group changes in $\OG^+(\Lten)$
under the specialization from $\Zgen$ to $Y$ (see Remark~\ref{rem:specialization}).
\begin{theorem}\label{thm:906608640}
Under an isometry $\SY\cong \Lten$ that maps $\PPP_Y$ to $\PPP_{10}$,
the  fundamental domain $\DY$
 in Theorem~\ref{thm:DY} is a union of following number of Vinberg chambers:
$$
2^{14}\cdot 3\cdot 5\cdot 7\cdot 17\cdot 31 = 906608640.
$$
\end{theorem}
This theorem also gives us a method to calculate a set of generators of $\Aut(\Zgen)$
explicitly. Using this set,
we carry out  an experiment on the entropies of automorphisms of $\Zgen$.
\par
The moduli of quartic Hessian surfaces has been studied by several authors
in order to investigate the moduli of cubic surfaces~(\cite{DvG},~\cite{Koike},~\cite{KondoH}).
In~\cite{DvG}, Dardanelli and  van Geemen studied several interesting subfamilies of this moduli.
It seems to be an interesting problem to investigate
the change of the automorphism group under  specializations of $Y$ to members of these subfamilies
by the method given in this paper.
\par
As is shown in Remark 4\,(2) of~\cite{MukaiOhashi},
there exists a specialization from $Y$ to
the Enriques surface $\YVI$
with $\Aut(\YVI)\cong \SSSS_5$ that appeared in the Nikulin-Kondo's classification
of Enriques surfaces with finite automorphism groups~(\cite{NikulinEnriques}, \cite{KondoFinite}).
Kondo pointed out that
the roots $\baru_{\alpha}$, $\barv_{\alpha}$ defining the walls
of $\DY$ given in Theorem~\ref{thm:DY} have the same configuration as
the smooth rational curves on $\YVI$.
(Compare~\eqref{eq:intnumbsuuww},~\eqref{eq:intnumbsuw}
and Figure 6.4 of~~\cite{KondoFinite}.)
It is also an interesting problem to investigate
the change of the automorphism group under various generalizations from
the seven Enriques surfaces with finite automorphism groups.
\par
The first application of  Borcherds method to the automorphism group
of an Enriques surface was given  in~\cite{Schiermonnikoog},
in which we investigated an Enriques surface whose universal cover is
a  $K3$ surface of Picard number $20$
with the transcendental lattice of discriminant $36$.
\par
This paper is organized as follows.
In Section~\ref{sec:preliminaries},
 we collect preliminaries about lattices and chambers.
 In Section~\ref{sec:unimodular},
 we recall the  results
 on the even unimodular hyperbolic lattices 
due to Vinberg~\cite{MR0422505} and Conway~\cite{MR690711}.
In Section~\ref{sec:Borcherds},
we explain Borcherds method
and its application  to $K3$ surfaces.
In Section~\ref{sec:Enriques},
we present some algorithms to study the geometry of an Enriques surface.
In particular,
we give an application of Borcherds method to an Enriques surface.
In Section~\ref{sec:QH},
we re-calculate, by the algorithm in~\cite{MR3456710},
 the results of Dolgachev and Keum~\cite{MR1897389}
on  the general quartic Hessian surface,
and convert these results into machine-friendly format.
With these preparations,
the main results
are proved in Section~\ref{sec:main}.
In the last section,
we calculate a set of generators of $\Aut(\Zgen)$,
and search for elements of $\Aut(\Zgen)$ with small entropies.
\par
For the computation,
we used {\tt GAP}~\cite{gap}.
The computational data are presented in the author's webpage~\cite{compdataHesseEnr}.
In fact,
once the basis of  the Leech lattice (Table~\ref{table:BasisLeech}),
the basis of $\SX$~\eqref{eq:basisEL},
the embedding $\SX\inj L_{26}$ (Table~\ref{table:embSXL}),
and  the basis of $\SY$ (Table~\ref{table:SXplus})
are fixed,
the other data can be derived by the algorithms in this paper.
\section{Preliminaries}\label{sec:preliminaries}
\subsection{Lattices}\label{subsec:lattices}
%
A submodule $M$ of a free $\Z$-module $L$ is said to be \emph{primitive}
if $L/M$ is torsion-free.
A non-zero vector $v\in L$ is  \emph{primitive} if $\Z v\subset L$
is primitive.
\par
A \emph{lattice} is a free $\Z$-module $L$ of finite rank with a non-degenerate symmetric bilinear form
$$
\intfvoid  \dcolon  L\times L\to \Z.
$$
Let $m$ be a non-zero  integer.
For an lattice $(L, \intfvoid)$,
we denote by $L(m)$
the lattice $(L, m \intfvoid)$.
Every vector of $L\tensor \R$ is written as a \emph{row} vector,
and
the \emph{orthogonal group} $\OG (L)$
of  $L$ acts on $L$ from the \emph{right}.
We put
$$
L\dual:=\Hom(L, \Z), \;\;\;  \LQ:=L\tensor\Q,\;\;\; \LR:=L\tensor\R.
$$
Then we have natural  inclusions $L\inj L\dual\inj \LQ\inj \LR$.
The \emph{discriminant group} of $L$ is defined to be $L\dual/L$.
A lattice $L$ is \emph{unimodular}
if $L\dual/L$ is trivial.
A lattice $L$ of rank $n$ is \emph{hyperbolic} if $n>1$ and the signature of 
 $\LR$ is $(1, n-1)$,
whereas  $L$ is \emph{negative-definite} if the signature is $(0, n)$.
\par
A lattice $L$ is \emph{even}
if $\intf{x, x}\in 2\Z$ for all $x\in L$.
Suppose that $L$ is even.
Then  the \emph{discriminant form}
$$
q_L \dcolon  L\dual/L \to \Q/2\Z
$$
of $L$
is defined by $q_L(x \bmod L):=\intf{x,x}\bmod 2\Z$ for $x\in L\dual$.
See~\cite{MR525944} for the basic properties of discriminant forms.
We denote by $\OG(q_L)$ the automorphism group of the finite quadratic form
$q_L$.
We regard $L\dual$ as a submodule of $L_{\Q}$,  and let $\OG(L)$ act on $L\dual$ from the right.
We have a natural homomorphism
$$
\eta_L \dcolon  \OG(L)\to \OG(q_L).
$$
\par
A vector $r\in L$ with $\intf{r,r}=-2$ is called a \emph{root}.
A root $r\in L$ defines a \emph{reflection}
$$
s_r \dcolon  x\mapsto x +\intf{x, r} r,
$$
which belongs to $\OG(L)$.
The \emph{Weyl group} $W(L)$ of $L$
is defined to be
the subgroup  of $\OG(L)$
generated by the reflections $s_r$ with respect to all the roots $r$ of $L$.
\par
Let  $L$ be an even hyperbolic lattice.
A \emph{positive cone} of $L$ is one of the two connected components
of $\shortset{x\in \LR}{\intform{x,x}>0}$.
We fix a positive cone $\PPP$ of $L$.
Let $\OG^+(L)$ denote the stabilizer subgroup of $\PPP$
in  $\OG(L)$.
Then $W(L)$ acts on $\PPP$.
For a root $r\in L$,
we put
$$
(r)\sperp:=\set{x\in \PPP}{\intf{x, r}=0}.
$$
The following is obvious:
\begin{proposition}\label{prop:roothyps}
The family $\shortset{(r)\sperp}{\textrm{$r$ is a root of $L$}}$
of hyperplanes of $\PPP$ is locally finite in $\PPP$.
\qed
\end{proposition}
%
A \emph{standard fundamental domain}
of the action of $W(L)$ on $\PPP$ is the closure
in $\PPP$ of a connected component of
$$
\PPP\;\;\setminus\;\;\bigcup_{r}\, (r)\sperp,
$$
where $r$ runs through the set of all roots.
Let $D$ be one of the  standard fundamental domains of $W(L)$.
We put
$$
\aut(D):=\set{g\in \OG^+ (L)}{D^g=D}.
$$
Then we have
$\OG^+ (L)=W(L)\rtimes\aut(D)$.
\subsection{$\VR$-chambers}\label{subsec:VRchambers}
Let $V$  be a $\Q$-vector space of dimension $n>0$,
and let $V\spstar$ denote the dual $\Q$-vector space $\Hom(V, \Q)$.
We put $\VR:=V\tensor\R$.
For a non-zero linear form  $f\in V\spstar\setminus\{0\}$,
we put
$$
H_f:=\shortset{x\in \VR}{f(x)\ge 0},\qquad
[f]\sperp:=\shortset{x\in \VR}{f(x)=0}=\partial H_f.
$$
%
\begin{definition}
A closed subset $\barC$ of  
$\VR$ is called
a \emph{$\VR$-chamber}
if $\barC$ contains a non-empty open subset of $\VR$, and
there exists a subset $\FFF$ of $V\spstar\setminus \{0\}$ such that
$$
\barC\;=\;  \bigcap_{f\in \FFF}{H_f}.
$$
When this is the case,
we say that  \emph{$\FFF$  defines the $\VR$-chamber $\barC$}.
\end{definition}
\par
Suppose that a subset $\FFF$ of $V\spstar\setminus\{0\}$
defines a $\VR$-chamber $\barC$.
We assume
the following:
\begin{equation}\label{eq:distinctassump}
H_f\ne H_{f\sprime}\;\;
\text{for  distinct $f, f\sprime\in \FFF$.}
\end{equation}
We say that an element $f$ of $V\spstar\setminus\{0\}$ \emph{defines a wall of $\barC$}
if  $\barC$ is contained in $H_f$ and $\barC\cap [f]\sperp$
contains a non-empty open subset of  $[f]\sperp$.
When this is the case,
we call $\barC\cap [f]\sperp$ the \emph{wall of $\barC$ defined by $f$}.
By the assumption~\eqref{eq:distinctassump},
we see that $f_0 \in  \FFF$ defines  a wall of $\barC$
if and only if there exists a point $x\in V$ such that
$$
f_0(x)<0,  \;\;\;\; \textrm{and}\;\; f(x)\ge 0\;\;\textrm{for all $f\in \FFF\setminus \{f_0\}$}.
$$
Hence we have the following. 
%
\begin{algorithm}\label{algo:wall}
Suppose that a $\VR$-chamber $\barC$
is defined by a finite subset $\FFF$ of  $V\spstar\setminus \{0\}$ satisfying~\eqref{eq:distinctassump}.
Then an element  $f_0 \in\FFF$ defines  a wall of $C$
if and only if the solution of the following problem of linear programing on $V$ over $\Q$
is unbounded to $-\infty$:
find the minimal value of  $f_0(x)$ subject to the constraints $ f(x)\ge 0$ for all $f\in \FFF\setminus \{f_0\}$.
\end{algorithm}
\par
Let $\barC$ and $\FFF$ be as above.
We define the \emph{faces of dimension $k$} of $\barC$
for $k=n-1, \dots, 1$ by descending induction on $k$.
The following is obvious.
\begin{lemma}
Suppose that $n>1$, and
 that $f_0\in\FFF$ defines a wall of $\barC$.
 For $g\in V\spstar$,
 let $g|_{ f_0\sperp}\colon f_0\sperp\to \Q$ denote the restriction of
 $g$
 to the hyperplane $f_0\sperp$ of $V$.
Then the wall $\barC\cap [f_0]\sperp$ of $\barC$ defined by $f_0$ is an $[f_0]\sperp$-chamber defined
by
$$
\FFF|_{f_0\sperp}:=\set{\,g|_{ f_0\sperp}}{g\in \FFF,\;\; g|_{ f_0\sperp}\ne 0}.
$$
\end{lemma}
The faces of $\barC$ of dimension $n-1$ are defined to be  the walls of $\barC$.
Suppose that $0< k<n-1$,
and let $F$ be a $(k+1)$-dimensional face of $\barC$.
Let $\gen{F}$ denote the minimal linear subspace of $V$ containing $F$.
We assume that
(i)
the linear space $\gen{F}$ is of dimension $k+1$,
(ii)
$F$ is equal to the closed subset  $\barC\cap\gen{F}$ of $\barC$,
and
(iii)
$F$ is an $(\gen{F}\tensor\R)$-chamber
defined by the subset
$$
\FFF|_{\gen{F}}:=\set{\; g|_{\gen{F}}}{\;g\in \FFF, \;\; g|_{\gen{F}}\ne 0}
$$
of $\Hom(\gen{F}, \Q)\setminus\{0\}$,
where $g|_{\gen{F}}$ is the restriction of $g$ to $\gen{F}$.
Then  the walls of the $(\gen{F}\tensor\R)$-chamber $F$ are defined.
A face of dimension $k$ of $\barC$ is defined to be a wall of a $(k+1)$-dimensional face of $\barC$.
It is obvious that $k$-dimensional faces satisfy the assumptions (i), (ii), (iii),
and hence the induction proceeds.
\par
When $\FFF$ is finite,
we can calculate all the faces of $\barC$
by using  Algorithm~\ref{algo:wall} iteratively.
\begin{remark}\label{rem:remove}
At every step of iteration,
we must remove redundant elements from $\FFF|_{\gen{F}}$
to obtain a subset  $\FFF\sprime_{\gen{F}}\subset \FFF|_{\gen{F}}$
that defines the walls of $F$ and satisfies~\eqref{eq:distinctassump}.
\end{remark}
\subsection{Chambers}\label{subsec:chambers}
Let $V$ be as in the previous subsection.
Suppose that $n>1$,  and that $V$ is
equipped with a non-degenerate symmetric bilinear form
$$
\intfvoid \dcolon  V\times V\to \Q
$$
such that
$\VR=V\tensor \R$ is of signature $(1, n-1)$.
By $\intfvoid$,
we identify $V$ and $V\spstar$.
In particular,
for a non-zero vector $v$ of $V$,
we put
$$
H_v:=\shortset{x\in \VR}{\intf{v, x}\ge 0},
\qquad
[v]\sperp:=\shortset{x\in \VR}{\intf{v, x}=0}=\partial H_v.
$$
Let $\PPPV$ be one of the two connected components
of $\shortset{x\in \VR}{\intf{x,x}>0}$, and
let $\closure{\PPP}_V$ denote the closure of $\PPP_V$ in $\VR$.
For a non-zero vector $v$ of $V$,
we put
$$
(v)\sperp:=[v]\sperp\cap \PPPV,
$$
which is non-empty if and only if $\intf{v,v}<0$.
%
\begin{definition}
A closed subset $C$ of $\PPPV$ is said to be  a \emph{chamber} if
there exists a subset $\FFF$ of $V\setminus \{0\}$
with the following properties.
\begin{enumerate}[(i)]
\item
The family $\shortset{(v)\sperp}{v\in \FFF, \intf{v,v}<0}$ of hyperplanes of the positive cone $\PPPV$
is locally finite in $\PPPV$.
\item
Under the identification $V=V\spstar$,
the  set $\FFF$ defines a $\VR$-chamber $\barC$ such that
\begin{equation}\label{eq:barCinPPP}
\barC\subset \closure{\PPP}_V
\;\; \textrm{and}\;\; C=\PPP_V\cap \barC.
\end{equation}
%
\end{enumerate}
When this is the case,
we say that the chamber $C$ is \emph{defined by $\FFF$}.
\end{definition}
\begin{remark}
A $\VR$-chamber $\barC$ satisfies $\barC\subset \closure{\PPP}_V$
if and only if $\barC\cap \PPP_V\ne \emptyset$ and
$\barC\cap \partial \,\closure{\PPP}_V$ is contained in the union of  one-dimensional faces of $\barC$.
\end{remark}
Let $C$ be a chamber defined by $\FFF\subset V\setminus \{0\}$.
Let $F$ be a $k$-dimensional face of $\barC$.
If $C\cap F\ne\emptyset$,
we say that $C\cap F$ is a \emph{face} of $C$
of dimension $k$.
Note that,
by~\eqref{eq:barCinPPP},
if $k>1$,
then $C\cap F$ is a face of $C$.
In particular,
since $n>1$,
if $\barC\cap [u]\sperp$ is a wall of $\barC$
defined by $u\in V\setminus\{0\}$,
then $C\cap (u)\sperp$ is  called the  \emph{wall of $C$ defined by $u$}.
\par
When $k=1$, we may have $C\cap F=\emptyset$.
\begin{definition}\label{def:idealface}
A one-dimensional face $F$ of $\barC$  contained in
$\barC \setminus C = \barC\cap \partial \,\closure{\PPP}_V$
is called an \emph{ideal face} of $C$.
By abuse of language, an ideal face of $C$ is also regarded as a face of $C$.
\end{definition}
%
%
%
\subsection{Chambers of a hyperbolic lattice}
Let $L$ be an even hyperbolic lattice with a positive cone $\PPP$.
Applying the above definition to $V=\LQ$,
we have the notion of chambers and their faces.
We mean by a \emph{chamber of $L$} a chamber of $L_{\Q}$.
\begin{definition}
We define the \emph{automorphism group of a chamber $C$ of $L$} by
$$
\aut(C):=\set{g\in \OG^+(L)}{C^g=C}.
$$
\end{definition}
\par
We put
$$
L\dual_{\prim}:=\set{v\in L\dual}{\textrm{$v$ is   primitive in $L\dual$}}.
$$
Then we have a canonical projection
$$
\LQ\setminus\{0\} \to L\dual_{\prim}, \quad v\mapsto \tilv
$$
such that $H_{v}=H_{\tilv}$ holds for all $v\in \LQ\setminus\{0\}$.
\begin{definition}
Let $C\cap (u)\sperp$ be a wall of a chamber $C$ of $L$ defined by $u\in \LQ$.
A vector $v\in L\dual$ is called  the \emph{primitive defining  vector} of the wall  $C\cap (u)\sperp$
if $v$ is the  vector of $L\dual_{\prim}$ satisfying $H_v=H_u$.
By definition,
each wall $C\cap (u)\sperp$ of $C$ has a unique primitive defining  vector $\tilde{u}$.
\end{definition}
If $\FFF\subset \LQ\setminus\{0\}$  defines a chamber $C$, then
so does the set
$$
\tilde{\FFF}:=\set{\tilv}{v\in \FFF}.
$$
The assumption~\eqref{eq:distinctassump}
holds automatically for $\tilde{\FFF}$.
Hence converting $\FFF$ to $\tilde{\FFF}$
is a convenient method
to achieve the property~\eqref{eq:distinctassump}
when we use Algorithm~\ref{algo:wall} iteratively to determine the faces of a chamber
(see Remark~\ref{rem:remove}).
\section{Even unimodular hyperbolic lattices $\Lten$ and $L_{26}$}\label{sec:unimodular}
For each positive integer $n$ with $n\equiv 2 \bmod 8$,
let $L_n$ denote an even unimodular hyperbolic lattice of  rank $n$,
which is unique up to isomorphism.
The lattice $L_2$ is denoted  by $U$.
We fix a basis $f_1, f_2$ of $U$
such that
 the Gram matrix of $U$ with respect to $f_1, f_2$ is
$$
\arraycolsep=4pt
\left[\begin{array}{cc} 0 &1 \\ 1 & 0\end{array}\right].
$$
\par
%
\subsection{The  lattice $\Lten$}\label{subsec:L10}
Let $E_8$ denote the \emph{negative-definite} even unimodular lattice of rank $8$
with the standard basis $e_1, \dots, e_8$,
whose intersection numbers are given by  the Dynkin diagram in Figure~\ref{fig:L10}.
%
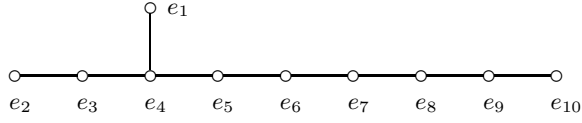
\begin{figure}
\def\ha{40}
\def\hav{37}
\def\hd{25}
\def\hdv{22}
\def\he{10}
\def\hev{7}
\setlength{\unitlength}{1.5mm}
{\small
\begin{picture}(80,11)(-5, 6)
\put(22, 16){\circle{1}}
\put(23.5, 15.5){$e\sb 1$}
\put(22, 10.5){\line(0,1){5}}
\put(9.5, \hev){$e\sb 2$}
\put(15.5, \hev){$e\sb 3$}
\put(21.5, \hev){$e\sb 4$}
\put(27.5, \hev){$e\sb 5$}
\put(33.5, \hev){$e\sb 6$}
\put(39.5, \hev){$e\sb 7$}
\put(45.5, \hev){$e\sb {8}$}
\put(51.5, \hev){$e\sb {9}$}
\put(57.5, \hev){$e\sb {10}$}
\put(10, \he){\circle{1}}
\put(16, \he){\circle{1}}
\put(22, \he){\circle{1}}
\put(28, \he){\circle{1}}
\put(34, \he){\circle{1}}
\put(40, \he){\circle{1}}
\put(46, \he){\circle{1}}
\put(52, \he){\circle{1}}
\put(58, \he){\circle{1}}
\put(10.5, \he){\line(5, 0){5}}
\put(16.5, \he){\line(5, 0){5}}
\put(22.5, \he){\line(5, 0){5}}
\put(28.5, \he){\line(5, 0){5}}
\put(34.5, \he){\line(5, 0){5}}
\put(40.5, \he){\line(5, 0){5}}
\put(46.5, \he){\line(5, 0){5}}
\put(52.5, \he){\line(5, 0){5}}
\end{picture}
}
\caption{Walls of the Vinberg chamber $D_{10}$}\label{fig:L10}
\end{figure}
We use $f_1, f_2, e_1, \dots, e_8$ as a basis of
$$
\Lten:=U\oplus E_8,
$$
and put
\begin{eqnarray*}\label{eq:w10}
e_9&:=& [ 1, 0, -3, -2, -4, -6, -5, -4, -3, -2 ], \\
e_{10}&:=& [ -1, 1, 0, 0, 0, 0, 0, 0, 0, 0 ], \\
w_{10}&:=&[ 31, 30, -68, -46, -91, -135, -110, -84, -57, -29 ].
\end{eqnarray*}
We have $\intf{e_9, e_9}=\intf{e_{10}, e_{10}}=-2$, $\intf{w_{10}, w_{10}}=1240$, and
\begin{equation*}\label{eq2:w10}
\WWW_{10}:=\set{r\in \Lten}{\intf{r, w_{10}}=1, \,\intf{r,r}=-2}=\{e_1, \dots, e_{10}\}.
\end{equation*}
The roots  $e_1, \dots, e_{10}$ form the Dynkin diagram in Figure~\ref{fig:L10}.
Let $\PPP_{10}$ be the positive cone of $\Lten$ containing $w_{10}$,
and $\barPPP_{10}$  the closure of $\PPP_{10}$ in $\Lten\tensor\R$.
We put
$$
\barD_{10}:=\set{x\in \Lten\tensor\R}{\intf{x, e_i}\ge 0 \;\;\textrm{for $i=1, \dots, 10$}\,\,},
\quad
D_{10}:=\barD_{10}\cap \PPP_{10}.
$$
%
\begin{theorem}[Vinberg~\cite{MR0422505}]\label{thm:L10}
The closed subset $D_{10}$ of $\PPP_{10}$ is a chamber.
The chamber $D_{10}$ is a standard fundamental domain of the action of $W(\Lten)$ on $\PPP_{10}$.
Each vector
$e_i$ of $\WWW_{10}$ defines a wall of $D_{10}$.
\end{theorem}
Since the diagram of the walls of $D_{10}$ in Figure~\ref{fig:L10} has no symmetries,
the automorphism group $\aut(D_{10})$ of  $D_{10}$ is trivial.
Hence $\OG^+(\Lten)$ is equal to  $W(\Lten)$,
which is generated by the reflections with respect to the roots $e_1, \dots, e_{10}$.
\begin{definition}
A standard fundamental domain of the action of $W(\Lten)$ on $\PPP_{10}$
is called a \emph{Vinberg  chamber}.
\end{definition}
%
%
\subsection{The  lattice $L_{26}$}\label{subsec:L26}
Let $\Lambda$ be the \emph{negative-definite} Leech lattice; that is,
the unique even \emph{negative-definite} unimodular lattice of rank $24$ with no roots.
As a basis of $\Lambda$,
we choose the row vectors $\lambda_1, \dots, \lambda_{24}$ of the  matrix given in Table~\ref{table:BasisLeech},
and consider them as elements of the quadratic space $\R^{24}$ with the negative-definite inner product
$$
(x, y)\mapsto -(x_1 y_1+\cdots+ x_{24} y_{24})/8.
$$
This basis is constructed from the extended binary Golay code
in the space of $\F_2$-valued functions
on $\P^1(\F_{23})=\{\infty\}\cup \F_{23}$,
where the value at $\infty$ is at the first coordinate of each row vector~(see Section 2.8 of~\cite{MR2977354}).
We put
$$
L_{26}:=U\oplus \Lambda.
$$
Then the vectors  $f_1, f_2, \lambda_1, \dots, \lambda_{24}$
form a basis of $L_{26}$, which we will use throughout this paper.
We  put
$$
w_{26}:=f_1,
\quad
\WWW_{26} := \set{r\in L_{26}}{\intf{r, w_{26}}=1, \,\intf{r, r}=-2}.
$$
Note that $\intf{w_{26}, w_{26}}=0$.
Let $\PPP_{26}$ be the positive cone of $L_{26}$
that contains $w_{26}$ in its closure $\barPPP_{26}$ in $L_{26}\tensor\R$.
We then put
$$
\barD_{26} := \set{x\in L_{26}\tensor\R}{\intf{x, r}\ge 0\;\;\textrm{for all $r\in \WWW_{26}$}\;},
\quad
D_{26}:=\barD_{26} \cap \PPP_{26}.
$$
\begin{theorem}[Conway~\cite{MR690711}]\label{thm:L26}
The closed subset $D_{26}$ of $\PPP_{26}$ is a chamber.
The chamber
$D_{26}$  is a standard fundamental domain of the action of $W(L_{26})$ on $\PPP_{26}$.
Each vector  $r_{\lambda}$ of $\WWW_{26}$ defines a wall of  $D_{26}$.
\end{theorem}
%
%
%
%
%
%
%
\begin{definition}
A standard fundamental domain of the action of $W(L_{26})$ on $\PPP_{26}$
is called a \emph{Conway  chamber}.
\end{definition}
A chamber $D\sprime$ in $\PPP_{26}$ is a Conway chamber
if and only if
$D\sprime$ is equal to $D_{26}^g$
for some  $g\in W(L_{26})$.
The \emph{Weyl vector} $w\sprime$ of a Conway chamber $D\sprime=D_{26}^g$ is defined to be $w_{26}^g$,
which is characterized by the following property:
$$
D\sprime=\set{x\in \PPP_{26}}{\intf{x, r}\ge 0 \;\textrm{for all roots $r\in L_{26}$ satisfying $\intf{w\sprime, r}=1$}\;}.
$$
\begin{table}
{
\tiny%
$$
\arraycolsep=4.2pt
\left[ \begin {array}{cccccccccccccccccccccccc}
0&8&0&0&0&0&0&0&0&0&0&0&0&0&0&0&0&0&0&0&0&0&0&0\\
0&4&0&0&0&0&0&0&0&0&0&0&0&0&0&0&0&0&0&0&4&0&0&0\\
0&4&0&0&0&0&0&0&0&0&0&0&0&0&0&0&4&0&0&0&0&0&0&0\\
0&4&0&0&0&0&4&0&0&0&0&0&0&0&0&0&0&0&0&0&0&0&0&0\\
4&4&0&0&0&0&0&0&0&0&0&0&0&0&0&0&0&0&0&0&0&0&0&0\\
0&4&0&0&4&0&0&0&0&0&0&0&0&0&0&0&0&0&0&0&0&0&0&0\\
0&4&0&0&0&0&0&4&0&0&0&0&0&0&0&0&0&0&0&0&0&0&0&0\\
2&2&0&0&2&0&2&2&0&0&2&0&0&0&0&0&2&0&0&0&2&0&0&0\\
0&4&4&0&0&0&0&0&0&0&0&0&0&0&0&0&0&0&0&0&0&0&0&0\\
0&4&0&0&0&0&0&0&0&0&0&0&0&0&0&0&0&0&0&0&0&4&0&0\\
0&4&0&0&0&0&0&0&0&0&0&0&0&0&0&4&0&0&0&0&0&0&0&0\\
0&2&2&0&0&0&2&0&0&0&0&0&0&0&0&2&2&0&0&0&2&2&2&0\\
0&4&0&0&0&0&0&0&0&0&0&0&4&0&0&0&0&0&0&0&0&0&0&0\\
2&2&2&0&2&2&0&0&0&0&0&0&2&0&0&0&0&0&0&0&2&2&0&0\\
2&2&2&0&0&0&0&2&0&0&0&0&2&0&0&2&2&2&0&0&0&0&0&0\\
2&2&2&0&0&0&2&0&0&0&2&0&2&0&2&0&0&0&0&0&0&0&2&0\\
0&4&0&4&0&0&0&0&0&0&0&0&0&0&0&0&0&0&0&0&0&0&0&0\\
2&2&2&2&0&0&0&0&0&0&2&2&0&0&0&0&2&0&0&0&0&2&0&0\\
2&2&2&2&2&0&2&0&0&0&0&0&0&0&0&2&0&0&2&0&0&0&0&0\\
2&2&2&2&0&0&0&2&2&0&0&0&0&0&0&0&0&0&0&0&2&0&2&0\\
2&0&2&2&0&0&2&0&0&0&0&0&2&0&0&0&2&0&0&0&2&0&0&2\\
0&0&2&2&0&2&0&0&0&0&0&2&2&0&0&0&0&0&0&2&0&2&0&2\\
0&0&2&2&0&0&0&0&0&2&0&0&2&0&0&2&0&2&2&0&0&0&0&2\\
1&-3&1&1&1&1&1&1&1&1&1&1&1&1&1&1&1&1&1&1&1&1&1&1
\end {array} \right]
$$
}
\caption{Basis of the Leech lattice $\Lambda$}\label{table:BasisLeech}
\end{table}
  \section{Borcherds method for $K3$ surfaces}\label{sec:Borcherds}
\subsection{Borcherds method}\label{subsec:Borcherds}
We review the method due to  Borcherds~(\cite{MR913200},~\cite{MR1654763})
to calculate a standard fundamental domain  of the Weyl group of an even hyperbolic lattice.
See also~\cite{MR3456710} for the algorithmic description of Borcherds method.
\begin{definition}
In the following, a \emph{tessellation} of a subset $N$ of a positive cone $\PPP$ of a hyperbolic lattice
means a decomposition of $N$ into a union of chambers such that,
if $D$ and $D\sprime$ are distinct chambers in the decomposition,
then the interiors of $D$ and $D\sprime$ are disjoint.
\end{definition}
Let $S$ be an even hyperbolic lattice.
Suppose that we have a primitive embedding
$$
i \dcolon  S\inj L_{26}
$$
of $S$ into the even unimodular hyperbolic lattice $L_{26}$ of rank $26$.
To simplify notation,
we use the same letter $i$ to denote
$i\tensor\R\colon S\tensor \R\inj L_{26}\tensor\R$.
We put
$$
\PPP_S:=i\inv(\PPP_{26}),
$$
where $\PPP_{26}$ is the fixed positive cone of $L_{26}$.
Then $\PPP_S$ is a positive cone of $S$.
We denote by
$$
\pr_S \dcolon  L_{26}\tensor\R\to S\tensor \R
$$
the orthogonal projection to $S\tensor \R$.
Note that $\pr_S(L_{26})\subset S\dual$.
Note also that,
for $v\in L_{26}\tensor\R$, we have
$$
i\inv ((v)\sperp)=(\pr_S(r))\sperp=\set{x\in \PPP_S}{\intf{x, \pr_S(v)}=0},
$$
and hence $i\inv ((v)\sperp) \ne \emptyset$ holds if and only if $\intf{\pr_S(v), \pr_S(v)}<0$.
We put
$$
\RRR_S:=\set{\pr_S(r)}{\textrm{$r$ is a root of $L_{26}$ such that $\intf{\pr_S(r), \pr_S(r)}< 0$  }}.
$$
Then the family of hyperplanes
$$
\HHH_S:=\set{(\pr_S(r))\sperp }{\pr_S(r) \in \RRR_S}
$$
is locally finite in $\PPP_S$
by Proposition~\ref{prop:roothyps}.
We investigate the  tessellation of $\PPP_S$ obtained by this family of hyperplanes.
\par
%
%
A Conway chamber $D\sprime=D_{26}^g$ in $\PPP_{26}$
is said to be \emph{non-degenerate} with respect to  $S$
if the closed subset $i\inv (D\sprime)$ of $\PPP_S$ contains a non-empty open subset of  $\PPP_S$.
When this is the case,
the closed subset $i\inv (D\sprime)$ of $\PPP_S$ is a chamber in $\PPP_S$ defined by the following subset of $\RRR_S$,
where $w\sprime$ is the Weyl vector of $D\sprime$;
$$
\set{\pr_S(r)}{\textrm{$r$ is a root of $L_{26}$ satisfying $\intf{r, w\sprime}=1$ and  $\intf{\pr_S(r), \pr_S(r)}< 0$}}.
$$
\begin{definition}
A chamber of $\PPP_S$ is said to be an \emph{induced chamber}
if  it is of the form
$i\inv (D\sprime)$,
where $D\sprime$ is a Conway chamber  non-degenerate
with respect to $S$.
\end{definition}
Since $\PPP_{26}$ is tessellated by Conway chambers,
the cone $\PPP_S$ is tessellated by induced chambers.
Each induced chamber is the closure in $\PPP_S$ of a connected component of
$$
\PPP_S\;\setminus\; \bigcup_{(v)\sperp\in \HHH_S}\, (v)\sperp.
$$
The following is the main result of~\cite{MR3456710}.
\begin{proposition}\label{prop:R} 
Suppose that the orthogonal complement of $S$ in $L_{26}$ cannot be embedded into
the negative-definite Leech lattice $\Lambda$.
Then each induced chamber $i\inv (D\sprime)$ has only a finite number of walls,
and the set of walls can be calculated from the Weyl vector $w\sprime$ of $D\sprime$.
\end{proposition}
\par
Every root $r$ of $S$ satisfies $r=\pr_S(r)$
and hence belongs to $\RRR_S$.
Therefore we have the following:
\begin{proposition}\label{prop:stdfundunion}
Each standard fundamental domain of the action of $W(S)$ on $\PPP_S$ is
tessellated by  induced chambers.
\end{proposition}
\par
For each wall of an induced chamber $i\inv (D\sprime)$,
 there exists a unique induced chamber that shares the wall with $i\inv (D\sprime)$.
More precisely, if $v\in \RRR_S$ defines a wall $i\inv (D\sprime) \cap (v)\sperp$ of  $i\inv (D\sprime)$,
then there exists a unique induced chamber $i\inv (D\spprime)$
such that $-v\in \RRR_S$ defines a wall $i\inv (D\spprime) \cap (-v)\sperp$ of $i\inv (D\spprime)$, and that we have
$$
i\inv (D\sprime) \cap (v)\sperp=i\inv (D\spprime) \cap (-v)\sperp.
$$
\begin{definition}\label{def:adjacent}
We call  the chamber $i\inv (D\spprime)$ satisfying the properties above the \emph{induced chamber adjacent to $i\inv (D\sprime)$
across the wall $i\inv (D\sprime) \cap (v)\sperp$}.
\end{definition}
\begin{definition}
We say that the primitive embedding $i\colon S\inj L_{26}$
is \emph{of simple Borcherds type}
if, for any two induced chambers $i\inv (D\sprime)$ and $i\inv (D\spprime)$,
there exists an isometry $g\in \OG^+(S)$ such that $i\inv (D\spprime)^g=i\inv (D\sprime)$.
\end{definition}
In order to prove that $i\colon S\inj L_{26}$
is of simple Borcherds type,
it is enough to choose an induced chamber $i\inv (D\sprime)$
and show that,
for each wall $i\inv (D\sprime)\cap (v)\sperp$ of $i\inv (D\sprime)$,
 there exists an element  $g\in \OG^+(S)$ extending to an isometry of $L_{26}$ such that $i\inv (D\sprime)^g$
is the induced chamber adjacent to $i\inv (D\sprime)$
across the wall $i\inv (D\sprime)\cap (v)\sperp$.
%
\subsection{Torelli theorem for  $K3$ surfaces}\label{subsec:Torell}
We recall how to apply  Torelli theorem~\cite{MR0284440} for complex $K3$ surfaces
to the study of the automorphism groups.
Let $X$ be a complex algebraic $K3$ surface.
We denote by $\SX$ the N\'eron-Severi lattice of $X$.
Suppose that $\rank \SX >1$.
Then $\SX$ is an even hyperbolic lattice.
For a divisor $D$ on $X$,
we denote $[D]\in \SX$ the class of $D$.
Let $\PPP_X$ be the positive cone of $\SX$ containing an ample class.
We put
$$
\nefbig(X):=\set{x\in \PPP_X}{\intf{x, [C]}\ge 0\;\;\textrm{for all curves $C$ on $X$}\;}.
$$
It is well-known that  $\nefbig(X)$ is a chamber of $\SX$,
that  $\nefbig(X)$ is a standard fundamental domain of the action of $W(\SX)$
on $\PPP_X$,
and that the correspondence $C\mapsto \nefbig(X)\cap  ([C])\sperp$
gives a bijection from
the set of
smooth rational curves $C$ on $X$
to the set of walls of the chamber $\nefbig(X)$.
\par
Let $\Aut(X)$ denote the automorphism group of $X$
acting on $X$ from the left.
Let $\aut(X)$ be the image
of the natural representation
$$
\varphi_X \dcolon  \Aut(X)\to \OG^+(\SX)
$$
defined by the \emph{pullback}
of the classes of divisors.
Then
we have $\aut(X)\subset \aut(\nefbig(X))$,
where $\aut(\nefbig(X))$ is the automorphism group of the chamber $\nefbig(X)$.
\par
Let $T_X$ denote the orthogonal complement of $\SX=H^2(X, \Z)\cap H^{1,1}(X)$
in the even unimodular lattice $H^2(X, \Z)$ with  the cup product,
and let
$\omega_X$ be a generator of the one-dimensional subspace
$H^{2,0}(X)$ of $T_X\tensor\C$.
We put
$$
\OG^\omega(T_X):=\set{g\in \OG(T_X)}{\omega_X^g\in \C\,\omega_X}.
$$
By~\cite{MR525944}, there exists a unique  isomorphism
$$
\sigma_X \dcolon  q_{S_X}\isom -q_{T_X}
$$
of finite quadratic forms such that the graph of $\sigma_X$ is the image of
  $H^2(X, \Z)\subset S_X\dual \oplus T_X\dual$
by the natural projection to $(S_X\dual /S_X)\oplus(T_X\dual/T_X)$.
We denote by
$$
\sigma_{X*} \dcolon  \OG(q_{S_X})\isom \OG(q_{T_X})
$$
the isomorphism induced by $\sigma_X$.
Recall that
$\eta_{L}\colon \OG(L) \to \OG(q_L)$ denotes the natural homomorphism.
By~\cite{MR525944} again,
an element $g$ of $\OG(\SX)$ extends to an isometry of $H^2(X, \Z)$ preserving the Hodge structure
if and only if
\begin{equation}\label{eq:etaeta}
\sigma_{X*}(\eta_{S_X} (g))\in \eta_{T_X}(\OG^\omega(T_X)).
\end{equation}
By  Torelli theorem for complex $K3$ surfaces~\cite{MR0284440},
we have the following:
\begin{theorem}\label{thm:aut}
{\rm (1)}
An element $g$ of $\aut(\nefbig(X))$ belongs to $\aut(X)$ if and only if  $g$ satisfies the condition~\eqref{eq:etaeta}.
{\rm (2)}
The kernel of $\varphi_X\colon \Aut(X)\to\aut(X)$ is isomorphic to the group
$$
\set{g\in \OG\sp{\omega} (T_X)}{\eta_{T_X}(g)=1}.
$$
\end{theorem}
Theorem~\ref{thm:aut} enables us to calculate $\Aut(X)$
from $\aut(\nefbig(X))$.
In~\cite{MR1618132},
Kondo studied the automorphism group of a generic Jacobian Kummer surface $X$
by finding a primitive embedding $S_X\inj L_{26}$ of simple Borcherds type and calculating $\aut(\nefbig(X))$.
Since then,
many authors have studied the automorphism groups of $K3$ surfaces by this method.
See~\cite{MR3456710} and the references therein.
On the other hand,
in~\cite{MR3286672} and~\cite{MR3456710}, this method was generalized to primitive embeddings $S_X\inj L_{26}$
that are \emph{not} of simple Borcherds type.
  \section{Computational  study of geometry of an Enriques surface}\label{sec:Enriques}
\subsection{Borcherds method for an Enriques surface}\label{subsec:appEnriques}
Suppose that a  $K3$ surface $X$ has an Enriques involution
$\enrinvol\in \Aut(X)$,
and let $g_{\enrinvol}:=\varphi(\enrinvol)\in \aut(X)$ denote its action on $S_X$.
Let $Y:=X/\gen{\enrinvol}$ be the quotient of $X$ by $\enrinvol$,
and let
$\pi\colon X\to Y$
denote the quotient morphism.
Let $S_Y$ denote the lattice of numerical equivalence classes of divisors on $Y$.
It is well-known that $S_Y$ is isomorphic to the even unimodular hyperbolic  lattice $\Lten$ of rank $10$.
Let $\PPP_Y$ be the positive cone of $S_Y$ containing an ample class.
We put
$$
\nefbig(Y):=\set{y\in \PPP_Y}{\intf{y, [C]}\ge 0\;\;\textrm{for all curves $C$ on $Y$}}.
$$
Let $\aut(Y)$ denote
the image of the natural representation
$$
\varphi_Y \dcolon  \Aut(Y)\to \OG^+(S_Y)
$$
defined by the \emph{pullback}  of the classes of divisors.
We  have
$\aut(Y)\subset \aut(\nefbig(Y))$.
We consider the following two primitive sublattices of $S_X$:
\begin{equation}\label{eq:SXpm}
\SX^+:=\shortset{v\in S_X}{v^{g_{\enrinvol}}=v},
\qquad
\SX^-:=\shortset{v\in S_X}{v^{g_{\enrinvol}}=-v}.
\end{equation}
%
The homomorphism $\pi^*\colon S_Y\to S_X$
induces an isomorphism of lattices
$$
\pi^* \dcolon  S_Y(2) \isom S_X^+,
$$
by which
we regard $S_Y$ as a $\Z$-submodule of $S_X$.
In particular, 
we have
$$
\PPP_Y=(S_Y\tensor\R)\cap \PPP_X,
\quad \nefbig(Y)=\PPP_Y\cap \nefbig(X).
$$
%
(The second quality follows from the projection formula and the fact that $\pi$ is finite.)
We use $\intfvoid_X$ and $\intfvoid_Y$ to denote the intersection forms of $S_X$ and $S_Y$,
respectively, so that, for $y, y\sprime\in S_Y\tensor \R$,  we have
$\intfX{y, y\sprime}=2\, \intfY{y, y\sprime}$.
\par
For a group $G$ and an element $g\in G$,
we denote by $Z_G(g)$ the centralizer of $g$ in $G$.
We have a natural isomorphism
\begin{equation*}\label{eq:ZAut}
\Aut(Y)\cong Z_{\Aut(X)}(\enrinvol)/\gen{\enrinvol}.
\end{equation*}
Hence we have the following:
\begin{proposition}\label{prop:Zaut}
Suppose that the representation  $\varphi_X\colon \Aut(X)\to \aut(X)$ is an isomorphism.
Then we have a  surjective homomorphism
\begin{equation*}\label{eq:Zaut}
\zeta \dcolon  Z_{\aut(X)}(g_\enrinvol)/\gen{g_\enrinvol} \surj \aut(Y)
\end{equation*}
defined by the commutative diagram
$$
\renewcommand{\arraystretch}{1.5}
\begin{array}{ccc}
Z_{\Aut(X)}(\enrinvol)/\gen{\enrinvol} & \cong &Z_{\aut(X)}(g_\enrinvol)/\gen{g_\enrinvol}\\
\wr\downarrow& & \mapdownright{\zeta} \\
\Aut(Y) &\displaystyle{\mathop{\longrightarrow \hskip -7pt \to}_{{\varphi_Y} }}& \aut(Y)\rlap{.}
\end{array}
$$
The  homomorphism  $\zeta$  is defined as follows.
If $g\in Z_{\aut(X)}(g_\enrinvol)$, then  $\SY^g=\SY$ holds, and the restriction
$g|_{\SY}\in \OG(\SY)$ of $g$ to $\SY$ gives $\zeta(g)\in \aut(Y)$.
\end{proposition}
\par
We denote the orthogonal projection to $S_Y\tensor\R=S_X^+\tensor\R$ by
$$
\pr^+ \dcolon  S_X\tensor\R\to S_Y\tensor\R.
$$
Suppose that we have a primitive embedding $i \colon S_X\inj L_{26}$,
and hence $\PPP_X$ is tessellated by induced chambers.
The composite of the primitive embeddings $\pi^*\colon S_Y(2)\inj S_X$ and  $i\colon S_X\inj L_{26}$
gives a primitive embedding of $S_Y(2)$ into $L_{26}$.
By this embedding,
the notion of induced chambers in $\PPP_Y$ is defined.
Recall that $\pr_S\colon L_{26}\tensor\R\to \SX\tensor \R$
is the orthogonal projection.  
Let
$$
\pr_S^+ \dcolon  L_{26}\tensor\R\to \SY\tensor \R
$$
denote the composite of $\pr_S$ and $\pr^+$.
Then the tessellation of $\PPP_Y$ by induced chambers is
 given by the locally finite family of hyperplanes
\begin{equation}\label{eq:prSplushamily}
\set{(\pr_S^+(r))\sperp}{\textrm{$r$ is a root of $L_{26}$ such that $\intfY{\pr_S^+(r), \pr_S^+(r)}<0$}}.
\end{equation}
This tessellation is the restriction to $\PPP_Y$  of the tessellation of $\PPP_X$
by induced chambers.
Since $N(X)$ is
a standard fundamental domain of the action of $W(\SX)$ on $\PPP_X$,
Proposition~\ref{prop:stdfundunion}
and $N(Y)=\PPP(Y)\cap N(X)$
imply the following:
\begin{proposition}
The chamber $N(Y)$ is tessellated by induced chambers. 
\end{proposition}
\par
Note that a chamber $\DY$ in $\PPP_Y$ is an induced chamber
if and only if there exists an induced chamber $\DX$
in $\PPP_X$ such that $\DY=\PPP_Y\cap \DX$.
More precisely, suppose that a subset $\FFF(\DX)$ of $S_X\dual\setminus\{0\}$ defines an induced chamber $\DX $ in $\PPP_X$;
$$
\DX =\set{x\in \PPP_X}{\intfX{v, x}\ge 0\;\;\textrm{for all $v\in \FFF(\DX)$}\;}.
$$
Then we have
$$
\DX \cap \PPP_Y =\set{y\in \PPP_Y}{\intfY{\pr^+(v), y}\ge 0\;\;\textrm{for all $v\in \FFF(\DX)$}\;}.
$$
Hence,
if  $\DX \cap \PPP_Y$
contains  a non-empty open subset of  $\PPP_Y$,
then $\DX \cap \PPP_Y$ is an induced  chamber
in $\PPP_Y$ defined by $\pr^+(\FFF(\DX))\setminus \{0\}$.
\subsection{Smooth rational curves on an Enriques surface}\label{subsec:ratcurvesY}
Let $X$ and $Y$ be as in the previous subsection.
In particular, $\SY$ is regarded as a $\Z$-submodule of $\SX$.
Let $\hY\in \SY$ be an ample class of $Y$. 
Then $\hY\in\SX$ is ample on $X$.
We explain a method to calculate the finite subset
$$
\RRR_d:=\set{[C]\in \SY}{\hbox{$C$ is a smooth rational curve on $Y$ such that $\intfY{[C], \hY}= d$}}
$$
of $\SY$ for each positive integer $d$
by induction on $d$ starting from $\RRR_0=\emptyset$.
Since the sublattice $\SX^-$ of $\SX$  is negative-definite,
the set
$$
\TTT:=\set{t\in \SX^-}{\intfX{t, t}=-4}
$$
is  finite and can be calculated.
We have the following:
\begin{lemma}[Nikulin~\cite{NikulinEnriques}]\label{lem:VVV}
Let $v\in \SY$ be a vector such that $\intfY{v, v}=-2$ and $\intfY{v, \hY}>0$.
Then the following conditions are equivalent.
\begin{enumerate}[{\rm (i)}]
\item The vector $v$ is the class of an effective divisor $D$ on $Y$ such that
$\pi^*D$
is written as $[\Delta]+[\Delta]^\enrinvol$,
where $\Delta$ is an effective divisor on $X$
satisfying $\intfX{[\Delta], [\Delta]^\enrinvol}=0$.
%
\item There exists an element $t\in \TTT$ satisfying $(v+t)/2\in \SX$.
\qed
%
%
\end{enumerate}
\end{lemma}
%
%
%
By the algorithm in Section 3 of~\cite{MR3166075},
we compute the finite set
$$
\VVV_d:=\set{v\in \SY}{\intfY{v, v}=-2, \; \intfY{v, \hY}=d}.
$$
We then compute
$$
\VVV\sprime_d:=\set{v\in \VVV_d}{\hbox{$v$ satisfies condition (ii) in Lemma~\ref{lem:VVV}}}.
$$
Lemma~\ref{lem:VVV} implies that $\RRR_d\subset \VVV\sprime_d$.
\begin{lemma}
A vector $v$ in $\VVV\sprime_d$ fails to belong to $\RRR_d$
if and only if there exists a vector $r\in \RRR_{d\sprime}$
with $0<d\sprime<d$ such that $\intfY{v, r}<0$.
\end{lemma}
\begin{proof}
Lemma~\ref{lem:VVV} implies that
there exists an effective divisor $D$ on $Y$
such that $v=[D]$.
Then $v\in \RRR_d$ if and only if $D$ is irreducible.
\par
Suppose that there exists a smooth rational curve $C$ on $Y$
such that $\intfY{v, [C]}<0$.
Then $D$ contains $C$.
In particular,
if $\intfY{[C], \hY}<\intfY{v, \hY}=d$,
then $D$ is not irreducible and hence $v\notin\RRR_d$.
Conversely, suppose that $D$ is reducible.
Let $\Gamma_1, \dots, \Gamma_N$ be the distinct reduced  irreducible components of  $D$.
If $\intfY{[D], [\Gamma_i]}\ge 0$ for all $i$,
we would have $\intfY{v, v}\ge 0$.
Therefore there exists an irreducible component $\Gamma_i$
such that $\intfY{[D], [\Gamma_i]}< 0$.
Then we have $\intfY{[\Gamma_i], [\Gamma_i]}<0$ and
hence $\Gamma_i$ is a smooth rational curve.
Since $D$ is reducible, we see that  $d\sprime:= \intfY{ [\Gamma_i], \hY}$
is smaller than $d$, and hence
 $r:=[\Gamma_i]\in \RRR_{d\sprime}$ satisfies $\intfY{v, r}<0$.
\end{proof}
%
%
%
\subsection{An elliptic fibration  on an Enriques surface}\label{subsec:ADEellfibY}
Let $Y$ be an Enriques surface with an ample class $\hY\in \SY$.
Let $\phi\colon Y\to\P^1$
be an elliptic fibration.
Then the class of a fiber of $\phi$ is written as $2 f_{\phi}$,
where $f_{\phi}$ is primitive in $\SY$.
For $p\in \P^1$, let $E_p$ denote the divisor on $Y$ such that
$$
\phi\inv(p)=\begin{cases}
E_p & \textrm{if $\phi\inv(p)$ is not a multiple fiber}, \\
2 E_p & \textrm{if $\phi\inv(p)$ is a multiple fiber}.
\end{cases}
$$
We give a method to calculate the reducible fibers of $\phi$.
Let $d_{\phi}:=\intfY{2f_{\phi}, \hY}$
be the degree of a fiber of $\phi$ with respect to $\hY$.
Then we obviously have
\begin{eqnarray*}
\RRR(\phi)&:=&\shortset{[C]\in \SY}{\hbox{$C$ is a smooth rational curve on $Y$ contained in a fiber of $\phi$}}\\
&=&\shortset{[C]\in \SY}{[C]\in \RRR_d\;\;\textrm{for some}\; d<d_{\phi}, \;\;\textrm{and}\; \;  \intfY{\,[C], f_{\phi}\,}=0}.
\end{eqnarray*}
We calculate the dual graph of the roots in $\RRR(\phi)$,
and decompose $\RRR(\phi)$
into equivalence classes according to
the connected components of the graph.
Then there exists a canonical bijection between the set
of these equivalence classes
and the set of reducible fibers of $\phi$.
Let $\varGamma_p\subset \RRR(\phi)$
be one of the equivalence classes,
and suppose that $\varGamma_p$
corresponds to a reducible fiber $\phi\inv(p)$.
The roots  in  $\varGamma_p$
form an indecomposable extended Dynkin diagram,
and its  $ADE$-type is the $ADE$-type of  the divisor $E_p$.
We calculate
$$
t(\varGamma_p):=\sum_{r\in \varGamma_p} \, m_r \cdot \intfY{r, \hY},
$$
where $m_r$ is
 the multiplicity in $E_p$ of the
 irreducible component  corresponding to $r\in \varGamma_p$
 (see Figure 1.8 of~\cite{MR2977354}).
We have
 $t(\varGamma_p)=\intfY{[E_p], \hY}$, and hence
 $t(\varGamma_p)$  is either $d_{\phi}$ or $d_{\phi}/2$.
Then $\phi\inv (p)$ is a multiple fiber
if and only if $t(\varGamma_p)= d_{\phi}/2$.
%
%
%
%
\section{General quartic Hessian surface}\label{sec:QH}
We review the results on   the general quartic Hessian surface
obtained in Dolgachev and Keum~\cite{MR1897389},
and re-calculate these results in the form of vectors and matrices.
\emph{From now on,
we denote by  $X$ the minimal resolution of the general quartic Hessian surface $\barX$
defined  in Introduction}.
\par
It is known that the N\'eron-Severi lattice $\SX$ of $X$ is of rank $16$,
and $\SX\dual/\SX$ is isomorphic to $(\Z/2\Z)^4\times (\Z/3\Z)$.
If two lines $\ell_{\beta}$ and $\ell_{\beta\sprime}$ on $\barX$ intersect,
the intersection point  is an ordinary node of $\barX$.
Hence we have
\begin{equation}
\intf{[E_{\alpha}], [E_{\alpha\sprime}]}=(-2) \cdot\delta_{\alpha \alpha\sprime},
\qquad
\intf{[L_{\beta}], [L_{\beta\sprime}]}=(-2) \cdot\delta_{\beta \beta\sprime},
\label{eq:EELL}
\end{equation}
where $\delta$ is Kronecker's delta symbol on $A\cup B$.
Recall that the indexing of $p_{\alpha}$ and $\ell_{\beta}$ was done in such a way that the following holds:
\begin{equation}
\intf{[E_{\alpha}], [L_{\beta}]}=
\begin{cases}
0 & \textrm{$\alpha\not\supset \beta$,}\\
1 & \textrm{$\alpha\supset \beta$.}
\end{cases}
\label{eq:EL}
\end{equation}
The lattice $\SX$ is generated by the classes $[E_{\alpha}]$ and $[L_{\beta}]$.
More precisely, the classes of the  following smooth rational curves form a basis of $\SX$:
\begin{eqnarray}\label{eq:basisEL}
&&
E_{123},
E_{124},
E_{125},
E_{134},
E_{135},
E_{145},
E_{234},
E_{235},
E_{245},
E_{345},\\
&&
L_{45},
L_{35},
L_{34},
L_{25},
L_{24},
L_{13}.\nonumber
\end{eqnarray}
We fix, once and for all, this basis,
and write elements of $S_X\tensor \R$ as \emph{row} vectors.
The Gram matrix of $\SX$
with respect to this basis is
readily calculated by~\eqref{eq:EELL}~and~\eqref{eq:EL}.
The fact that the classes of the $16$ curves above form a basis of $\SX$
can be confirmed by checking that the determinant of this Gram matrix is equal to $-2^4\cdot 3$.
%
%
%
Let $h_Q\in \SX$ denote the class of the pullback of a hyperplane section of $\barX\subset \P^3$
by the minimal resolution $X\to \barX$.
Then we have
$$
h_Q=[ -1, 1, 1, -1, -1, 1, 1, 1, 3, 1, 2, 0, 0, 2, 2, -2 ].
$$
%
\begin{proposition}\label{prop:period}
{\rm (1)}
An element $g$ of $\aut(\nefbig(X))$ belongs to $\aut(X)$ if and only if $g$ satisfies
$\eta_{S_X} (g)=\pm 1$, where $\eta_{S_X}\colon \OG(\SX)\to\OG(q_{\SX})$
is the natural homomorphism.
{\rm (2)}
The representation $\varphi_X\colon \Aut(X)\to \aut(X)$ is an isomorphism.
\end{proposition}
\begin{proof}
Since $\barX$ is \emph{general} in the family of quartic Hessian surfaces,
we have
$$
\OG^\omega(T_X)=\{\pm 1\}.
$$
Therefore the statements  follow from Theorem~\ref{thm:aut}.
\end{proof}
\par
We then apply Borcherds method to $\SX$.
Recall that we have fixed a basis
$f_1, f_2, \lambda_1, \dots, \lambda_{24}$
of $L_{26}$ in Section~\ref{subsec:L26}.
Let $M$ be the matrix
given in Table~\ref{table:embSXL}.
\begin{table}
{
\tiny%
\renewcommand{\arraystretch}{2}
\vskip -.7cm
\centering
\rotatebox{90}{
\begin{minipage}{12.5cm}
$$
\arraycolsep=1.5pt
\left[\begin{array}{cccccccccccccccccccccccccc}
\;\;1\;\; & \;1\; & 0 & 0 & 0 & 0 & 0 & 0 & 0 & 0 & 0 & 0 & 0 & 0 & 0 & 1 & 0 & 0 & 0 & 0 & 0 & 0 & 0 & 0 & 0 & 0 \\
1 & 1 & -6 & 2 & 2 & 2 & 2 & 2 & 2 & -2 & 1 & 0 & 1 & -1 & 0 & 0 & -1 & 1 & 0 & 1 & -1 & 0 & -1 & 0 & 1 & 0 \\
1 & 1 & 6 & -2 & -2 & -2 & -1 & -2 & 0 & 1 & 0 & -1 & -1 & 1 & 0 & 1 & 0 & -1 & 0 & 0 & 1 & -1 & 1 & -1 & -1 & 2 \\
1 & 1 & 0 & 0 & 0 & 0 & 0 & 0 & -2 & 2 & 1 & 1 & 0 & -1 & 0 & -1 & 1 & 0 & 0 & -1 & 0 & 1 & -1 & 1 & 0 & 0 \\
1 & 1 & 0 & 0 & 0 & 0 & 1 & 0 & 0 & 0 & 0 & -1 & 0 & 1 & 0 & 0 & 0 & 0 & 0 & 0 & 0 & 0 & -1 & 1 & 0 & 0 \\
1 & 1 & -1 & 0 & 0 & 0 & 1 & 0 & 0 & 1 & 1 & 1 & 0 & 0 & 0 & -1 & 0 & 0 & 0 & -1 & 0 & 0 & 0 & 1 & 0 & 0 \\
1 & 1 & 5 & -2 & -2 & -2 & 0 & 0 & 0 & 0 & 0 & -1 & -1 & 2 & 0 & 0 & 0 & -1 & 0 & 0 & 0 & -1 & 2 & -1 & -1 & 2 \\
1 & 1 & -3 & 0 & 2 & 2 & 1 & 0 & 0 & 0 & 1 & 0 & 1 & -1 & 0 & 1 & -1 & 0 & 0 & 0 & -1 & 1 & -1 & 0 & 1 & 0 \\
1 & 1 & 0 & 0 & 0 & 0 & 0 & 0 & 0 & 0 & 0 & 0 & 0 & 0 & 0 & 0 & 1 & 0 & 0 & 0 & 0 & 0 & 0 & 0 & 0 & 0 \\
1 & 1 & -3 & 2 & 2 & 0 & 1 & 0 & 0 & 0 & 1 & 1 & 0 & -1 & 0 & -1 & 0 & 1 & 0 & -1 & 1 & 0 & -1 & 1 & 0 & 0 \\
1 & 1 & 0 & -1 & 0 & 0 & 1 & 0 & 0 & 1 & 0 & -1 & 0 & 1 & 0 & 1 & -1 & 0 & 1 & 0 & -1 & 0 & -1 & 0 & 1 & 0 \\
1 & 1 & 3 & -1 & -2 & -1 & 0 & 0 & 0 & 0 & 0 & 0 & -1 & 1 & 0 & 0 & 1 & -1 & 1 & 0 & 0 & -1 & 1 & -1 & -1 & 2 \\
1 & 1 & 0 & 0 & 0 & 1 & 0 & 0 & 1 & -1 & -1 & -1 & 0 & 0 & 0 & 1 & 0 & 0 & 0 & 1 & 0 & 0 & 0 & 0 & 0 & 0 \\
1 & 1 & -1 & 1 & 0 & 0 & 1 & 1 & 0 & 0 & 0 & 0 & 0 & 0 & 0 & -1 & 0 & 0 & 0 & 0 & 0 & 0 & 0 & 1 & 0 & 0 \\
1 & 1 & 0 & 0 & 1 & 0 & 0 & -1 & -1 & 1 & 0 & 1 & 0 & -1 & 0 & 0 & 0 & 0 & 0 & -1 & 1 & 1 & -1 & 1 & 0 & 0 \\
1 & 1 & -1 & 0 & 1 & 0 & 1 & 0 & 1 & 0 & 0 & 0 & 0 & 0 & 0 & 0 & -1 & 0 & 0 & 0 & 0 & 0 & 0 & 0 & 1 & 0
\end{array}\right]
$$
\end{minipage}
}
}
\vskip -.7cm
\caption{Primitive embedding  of $\SX$ into $L_{26}$}\label{table:embSXL}
\end{table}
It is easy see that  the homomorphism
$i\colon S_X\to L_{26}$ given by
$v\mapsto vM$ is a primitive embedding of  $S_X$ into
 $L_{26}$
that maps $\PPP_X$ into $\PPP_{26}$.
From now on,
we regard $\SX$ as a primitive sublattice of $L_{26}$ by this embedding $i$.
\begin{remark}
This embedding $i\colon S_X\to L_{26}$ is equal to the embedding
given by Dolgachev and Keum~\cite{MR1897389}
up to the action of $\OG^{+}(S_X)$ and $\OG^{+}(L_{26})$.
See~\cite{MR3456710}
for a general method to embed the N\'eron-Severi lattice of a $K3$ surface into $L_{26}$
in Borcherds method.
\end{remark}
As in Section~\ref{subsec:Borcherds},
we denote by
$\pr_S\colon L_{26}\tensor\R \to \SX\tensor\R$
the orthogonal projection.
We can calculate a Gram matrix of
the orthogonal complement $R$ of $S_X$ in $L_{26}$ explicitly,
and confirm that
$R$  contains a root.
Hence $R$ cannot be embedded into the negative-definite Leech lattice $\Lambda$.
Therefore Proposition~\ref{prop:R} can be applied.
\par
Let $D_{26}$ be the Conway chamber
with the Weyl vector $w_{26}=f_1\in L_{26}$
(see Section~\ref{subsec:L26}).
We put
$$
\DX :=i\inv (D_{26}).
$$
%
\begin{proposition}\label{prop:DX}
The closed subset   $\DX $  of $\PPP_X$ contains
the vector
$$
\hX:=\pr_S(w_{26})=[ -3, 2, 2, -3, -3, 2, 2, 2, 7, 2, 5, 0, 0, 5, 5, -5 ]
$$
of square-norm $20$ in its interior.
Moreover $\hX$ belongs to $\nefbig(X)$.
Hence $\DX $ is an induced chamber contained in $\nefbig(X)$.
\end{proposition}
\begin{proof}
We can calculate the set
$$
\FFF\sprime(\DX ):=\set{\pr_S(r)}{r\in \WWW_{26},\;\; \intf{\pr_S(r), \pr_S(r)}< 0\;}
$$
by the algorithm presented in Section 5 of~\cite{MR3456710}. 
By definition,
we have
$$
\DX =\set{x\in \PPP_X}{\intf{x, v}\ge 0\;\;\textrm{for all $v\in \FFF\sprime(\DX )$}\,}.
$$
Hence we  can confirm 
that $\hX$
is an interior point of $\DX $.
Therefore $\DX $ is an induced chamber.
Moreover, we can confirm that
$$
\set{r\in S_X}{\intf{r, \hX}=0,  \,\intf{r,r}=-2}=\emptyset
$$
by the algorithm given in Section 3 of~\cite{MR3166075}.
Since the singularities of $\barX$ consist only of the ordinary nodes $p_{\alpha}$ ($\alpha\in A$), we have
$$
\set{r\in S_X}{\intf{r, h_Q}=0,  \, \intf{r,r}=-2}=\set{\pm [E_{\alpha}]}{\alpha\in A},
$$
and the classes $[E_{\alpha}]$ define the walls of $\nefbig(X)$ that pass through
the half-line  $\R_{\ge 0}\, h_Q$.
We see that
$\intf{[E_{\alpha}], \hX}=1>0$ for all $\alpha\in A$, and confirm that
$$
\set{r\in S_X}{\intf{r, h_Q}>0, \, \intf{r, \hX}<0, \, \intf{r,r}=-2}=\emptyset
$$
by the algorithm given in Section 3 of~\cite{MR3166075}.
These facts  imply that  $\hX$ is an interior point of $\nefbig(X)$.
By Proposition~\ref{prop:stdfundunion},
$\nefbig(X)$ is tessellated by  induced chambers.
Hence $\DX $ is contained in $\nefbig(X)$.
\end{proof}
\begin{proposition}\label{prop:wallsDX}
The number of walls of $\DX $ is
$20+10+24+30=84$,
among which $20$ walls  are walls of $\nefbig(X)$
and they are defined by the roots $[E_{\alpha}]$ and $[L_{\beta}]$,
whereas the other $10+24+30$ walls are not walls of $\nefbig(X)$.
\end{proposition}
\begin{proof}
In the proof of Proposition~\ref{prop:DX},
the set $\FFF\sprime(\DX )$ of vectors defining the chamber $\DX $ is calculated.
From $\FFF\sprime(\DX )$,
we can calculate the set $\FFF(\DX )$ of primitive defining vectors $v$ of walls of $\DX $ by Algorithm~\ref{algo:wall}.
The result is given in Table~\ref{table:DXwalls}.
\begin{table}
$$
\arraycolsep=5pt
\begin{array}{ccccc}
\textrm{type} & \intf{v, \hX} & \intf{v,v} & \textrm{number} & \mystrutd{4pt} \\
\hline
\textrm{(a)} & 1 & -2 & 20 &   \textrm{outer}  \mystruth{12pt}\\
\textrm{(b)} & 2 & -1 & 10 &   \textrm{inner} \\
\textrm{(c)} & 5 & -2/3 & 24 &   \textrm{inner} \\
\textrm{(d)} & 4 & -2/3 & 30 &   \textrm{inner} \\
\end{array}
$$
\vskip 1mm
\caption{Walls of $\DX $}\label{table:DXwalls}
\end{table}
The walls of $\DX $ are divided into four types (a)--(d)
according to the values of $\intf{v, \hX}$ and $\intf{v,v}$,
where $v\in \SX\dual$ is the primitive defining  vector of the wall.
It turns out that
the $20$ walls of type (a) are defined by
$[E_{\alpha}]$ and $[L_{\beta}]$,
and hence these walls are walls of $\nefbig(X)$.
For each of the other walls,
there exists no positive integer $k$ such that $k^2 \intf{v,v} =-2$.
Hence these walls are not walls of $\nefbig(X)$.
\end{proof}
We call  a wall of $\DX $ an \emph{outer wall} if it is a wall of $\nefbig(X)$,
and call it an \emph{inner wall} otherwise. 
\begin{proposition}\label{prop:DXg}
If $g\in\aut(X)$,
then $\DX ^g$ is an induced chamber.
\end{proposition}
\begin{proof}
Let $R$ denote the orthogonal complement of $S_X$ in $L_{26}$.
By~\cite{MR525944}, we have an isomorphism
$\sigma_R\colon q_{R}\isom -q_{S_X}$
of finite quadratic forms
given by the even unimodular overlattice $L_{26}$ of $S_X\oplus R$.
Let $\sigma_{R*}\colon \OG(q_{R})\isom \OG(q_{S_X})$ denote  the isomorphism
induced by $\sigma_R$.
By Proposition~\ref{prop:period},
we have $\eta_{S_X}(g)=\pm 1$ and hence $\eta_{S_X}(g)$ belongs to the image of
the composite homomorphism
$$
 \OG(R)\maprightsp{\eta_R}  \OG(q_{R}) \maprightsp{\hskip -3pt \sigma_{R*}} \OG(q_{S_X}).
$$
By~\cite{MR525944} again,  there exists an element $\tilde{g}$ of $\OG(L_{26})$
such that  $\tilde{g}$ preserves
the primitive sublattices $S_X$ and $R$
and that the restriction of $\tilde{g}$ to $S_X$ is equal to $g$.
Since $\DX ^g=i\inv (D_{26}^{\tilde{g}})$,
we see that $\DX ^g$ is an induced chamber.
\end{proof}
\begin{proposition}\label{prop:autSX}
The automorphism group $\aut(\DX )$ of the chamber $\DX $ is of order $240$,
and we have $\hX^g=\hX$ for all $g\in \aut(\DX )$.
%
%
\end{proposition}
\begin{proof}
If $g\in \OG^+(S_X)$ induces an automorphism of $\DX $,
then $g$ induces a permutation of the set
$\{([E_{\alpha}])\sperp, ([L_{\beta}])\sperp\}$ of walls of type (a),
and hence
 $g$ induces a permutation of 
$A\cup B$ 
that preserves the intersection numbers~\eqref{eq:EELL} and~\eqref{eq:EL}.
The permutations of $A\cup B$
preserving~\eqref{eq:EELL} and~\eqref{eq:EL}
form a group  of order $240$ generated by
the permutations of $\{1, \dots, 5\}$ and the switch $\alpha=\bar{\beta}\leftrightarrow\beta=\bar{\alpha}$.
Conversely, each of  these $240$ permutations induces an isometry of $S_X$ that preserves $\DX $.
By direct calculation, we see that
\begin{equation}\label{eq:hXsum}
\hX=\sum_{\alpha\in A}[E_{\alpha}]+\sum_{\beta\in B}[L_{\beta}].
\end{equation}
Hence we have $\hX^g=\hX$ for all $g\in \aut(\DX )$.
\end{proof}
%
%
\begin{proposition}\label{prop:enrinvol}
The group $\aut(\DX )\cap\aut(X)$
is of order $2$, and
its non-trivial element $g_\enrinvol$ 
is the image $\varphi_X (\enrinvol)$
of an Enriques involution
$\enrinvol\in \Aut(X)$
by the natural representation $\varphi_X\colon \Aut(X)\to \aut(X)$.
This Enriques involution $\enrinvol$
switches  $E_{\alpha}$ and $L_{\bar{\alpha}}$ for each $\alpha\in A$.
\end{proposition}
\begin{proof}
By means of Proposition~\ref{prop:period}\,(1),
we can check by direct calculation that $\aut(\DX )\cap\, \aut(X)$
consists of the identity and the isometry that comes from the switch
$\alpha\leftrightarrow\bar{\alpha}$.
The matrix presentation of the isometry $g_{\enrinvol}\in \aut(\DX )$ induced by the switch
is given in Table~\ref{table:enrinvol}.
As in~\eqref{eq:SXpm},
we put
$$
S_X^+:=\shortset{v\in S_X}{v^{g_{\enrinvol}}=v},
\qquad
S_X^-:=\shortset{v\in S_X}{v^{g_{\enrinvol}}=-v}.
$$
Then $S_X^+$ is of rank $10$ generated by the row vectors $\eta_1, \dots, \eta_{10}$ of the matrix given in Table~\ref{table:SXplus}.
%
%
We see that $S_X^+\cong \Lten(2)$.
Indeed, we have chosen the basis $\eta_1, \dots, \eta_{10}$ of $S_X^+$ in such a way that
the homomorphism from $\Lten$ to $\SX^+$ given by
\begin{equation}\label{eq:homtoetas}
f_1\mapsto \eta_1,\quad
f_2\mapsto \eta_2,  \quad
e_1\mapsto  \eta_3, \quad  \dots, \quad  e_8\mapsto \eta_{10}
\end{equation}
induces an isometry $\Lten(2)\isom S_X^+$,
where $f_1, f_2, e_1, \dots, e_8$ are the basis of  $\Lten$ fixed in Section~\ref{subsec:L10}.
On the other hand, we can confirm 
that the negative-definite even lattice $S_X^-$  has no roots~(see Proposition~\ref{prop:KondoSXminus}).
Hence we conclude  that $g_{\enrinvol}$ is the image $\varphi_X(\enrinvol)$ of an Enriques involution $\enrinvol$ on $X$
by the criterion given in~\cite{MR1060704}.
\end{proof}
\begin{table}
{
\tiny%
$$
 \arraycolsep=4pt
\left[\begin{array}{cccccccccccccccc}
0 & 0 & 0 & 0 & 0 & 0 & 0 & 0 & 0 & 0 & 1 & 0 & 0 & 0 & 0 & 0 \\
0 & 0 & 0 & 0 & 0 & 0 & 0 & 0 & 0 & 0 & 0 & 1 & 0 & 0 & 0 & 0 \\
0 & 0 & 0 & 0 & 0 & 0 & 0 & 0 & 0 & 0 & 0 & 0 & 1 & 0 & 0 & 0 \\
0 & 0 & 0 & 0 & 0 & 0 & 0 & 0 & 0 & 0 & 0 & 0 & 0 & 1 & 0 & 0 \\
0 & 0 & 0 & 0 & 0 & 0 & 0 & 0 & 0 & 0 & 0 & 0 & 0 & 0 & 1 & 0 \\
-2 & 1 & 1 & -2 & -2 & 1 & 0 & 0 & 3 & 0 & 2 & -1 & -1 & 2 & 2 & -3 \\
-1 & 1 & 0 & -1 & -2 & 0 & 1 & 0 & 2 & 0 & 1 & -1 & 0 & 1 & 2 & -2 \\
-1 & 0 & 1 & -2 & -1 & 0 & 0 & 1 & 2 & 0 & 1 & 0 & -1 & 2 & 1 & -2 \\
0 & 0 & 0 & 0 & 0 & 0 & 0 & 0 & 0 & 0 & 0 & 0 & 0 & 0 & 0 & 1 \\
0 & -1 & -1 & 1 & 1 & 0 & 0 & 0 & -1 & 1 & 0 & 1 & 1 & -1 & -1 & 1 \\
1 & 0 & 0 & 0 & 0 & 0 & 0 & 0 & 0 & 0 & 0 & 0 & 0 & 0 & 0 & 0 \\
0 & 1 & 0 & 0 & 0 & 0 & 0 & 0 & 0 & 0 & 0 & 0 & 0 & 0 & 0 & 0 \\
0 & 0 & 1 & 0 & 0 & 0 & 0 & 0 & 0 & 0 & 0 & 0 & 0 & 0 & 0 & 0 \\
0 & 0 & 0 & 1 & 0 & 0 & 0 & 0 & 0 & 0 & 0 & 0 & 0 & 0 & 0 & 0 \\
0 & 0 & 0 & 0 & 1 & 0 & 0 & 0 & 0 & 0 & 0 & 0 & 0 & 0 & 0 & 0 \\
0 & 0 & 0 & 0 & 0 & 0 & 0 & 0 & 1 & 0 & 0 & 0 & 0 & 0 & 0 & 0
\end{array}\right]
$$
}
\caption{Matrix representation of the Enriques involution $\enrinvol$}\label{table:enrinvol}
\end{table}
\begin{table}
{
\tiny%
$$
 \arraycolsep=4pt
\left[\begin{array}{cccccccccccccccc}
-1 & 1 & 0 & -1 & -1 & 1 & 1 & 0 & 2 & 1 & 2 & 0 & 0 & 1 & 2 & -2 \\
-1 & 1 & 1 & -1 & -1 & 0 & 1 & 1 & 2 & 0 & 1 & 0 & 0 & 2 & 2 & -2 \\
0 & 0 & 0 & 1 & 0 & 0 & 0 & 0 & 0 & 0 & 0 & 0 & 0 & 1 & 0 & 0 \\
-2 & 1 & 1 & -2 & -2 & 1 & 0 & 1 & 2 & 0 & 1 & 0 & -1 & 2 & 1 & -3 \\
1 & -1 & -1 & 2 & 1 & -1 & 1 & -1 & -2 & 1 & -1 & 0 & 2 & -2 & -1 & 2 \\
-1 & 1 & 1 & -3 & -2 & 1 & 0 & 1 & 3 & -1 & 2 & -1 & -2 & 2 & 2 & -3 \\
1 & -1 & -1 & 2 & 3 & -1 & -1 & 0 & -3 & 1 & -2 & 2 & 1 & -2 & -2 & 3 \\
1 & 0 & 0 & 1 & 0 & 0 & 0 & -1 & -1 & 0 & 0 & 0 & 1 & -1 & -1 & 1 \\
-2 & 1 & 1 & -2 & -2 & 1 & 1 & 0 & 3 & -1 & 1 & -2 & -1 & 2 & 3 & -3 \\
1 & 0 & -1 & 1 & 1 & -1 & 0 & 0 & -1 & 0 & -1 & 1 & 0 & -1 & -1 & 2
\end{array}\right]
$$
}
\caption{Basis of $S_X^+$}\label{table:SXplus}
\end{table}
%
%
%
%
\par
The  walls of $\DX $ of type (b) play an important role in the
study of the Enriques surface $Y:=X/\gen{\enrinvol}$.
For $\alpha\in A$,
let $v_{\alpha}$ be the unique vector of $S_X\dual$ that satisfies
\begin{equation}\label{eq:valpha}
\intf{v_{\alpha}, [E_{\alpha\sprime}]}=\delta_{\alpha \alpha\sprime},
\;\;
\intf{v_{\alpha}, [L_{\beta}]}=\delta_{\bar{\alpha}\beta}.
\end{equation}
Then the walls of $\DX $ of type (b) are exactly the walls defined by $v_{\alpha}$ for some $\alpha\in A$.
Since $g_{\enrinvol}$ comes from the switch $\alpha\leftrightarrow \bar{\alpha}$,
we have
\begin{equation}\label{eq:valphagepsilon}
v_{\alpha}^{g_{\enrinvol}}=v_{\alpha}.
\end{equation}
\begin{proposition}\label{prop:iotaalpha}
For $\alpha\in A$, let $\bar{\pi}_{\alpha}\colon \barX\to \P^2$
be the projection from the center $p_{\alpha}\in \barX$,
and let $\iota_{\alpha}\colon X\to X$ be the involution
obtained from the double covering
 $\bar{\pi}_{\alpha}$.
Then $g_{\alpha}:=\varphi_X(\iota_{\alpha})\in \aut(X)$ maps $\DX $
to the induced chamber adjacent to $\DX $ across the wall $\DX \cap (v_{\alpha})\sperp$ of type {\rm (b)}.
Moreover $\iota_{\alpha}$  commutes with $\enrinvol$.
\end{proposition}
\begin{proof}
Let $\pi_{\alpha}\colon X\to \P^2$ denote the composite of  $\bar{\pi}_{\alpha}$
with  the minimal resolution $X\to\barX$.
Then the class $\pi_{\alpha}^*([l])\in \SX$
of the pullback of a line $l$ on $\P^2$ is $h_Q-[E_{\alpha}]$.
We calculate the finite  set
$$
\tilde{\Gamma}_{\alpha}:=\set{r\in \SX}{\intf{r, \pi_{\alpha}^*([l])}=0, \, \intf{r, r}=-2, \, \intf{r, \hX}>0\; }
$$
by the algorithm given in Section 3 of~\cite{MR3166075}.
From $\tilde{\Gamma}_{\alpha}$ and using the ample class $\hX$ of $X$,
we can calculate by the method described in Section 6.1 of~\cite{Schiermonnikoog}
the set ${\Gamma}_{\alpha}$ of classes of smooth rational curves
that are contracted by $\pi_{\alpha}$.
It turns out that the vectors in ${\Gamma}_{\alpha}$ form
the following Dynkin diagram.
$$
\def\ha{40}
\def\hav{37}
\def\hd{25}
\def\hdv{22}
\def\he{7}
\def\hev{10}
\setlength{\unitlength}{1.5mm}
\begin{picture}(80,9)(6, 5)
\put(10, \he){\circle{1}}
\put(16, \he){\circle{1}}
\put(22, \he){\circle{1}}
\put(8.5, \hev){$E_{\alpha_{11}}$}
\put(14.5, \hev){$L_{\beta_{1}}$}
\put(20.5, \hev){$E_{\alpha_{12}}$}
\put(10.5, \he){\line(5, 0){5}}
\put(16.5, \he){\line(5, 0){5}}
\put(29, \he){\circle{1}}
\put(35, \he){\circle{1}}
\put(41, \he){\circle{1}}
\put(27.5, \hev){$E_{\alpha_{21}}$}
\put(33.5, \hev){$L_{\beta_{2}}$}
\put(39.5, \hev){$E_{\alpha_{22}}$}
\put(29.5, \he){\line(5, 0){5}}
\put(35.5, \he){\line(5, 0){5}}
\put(46.5, \hev){$E_{\alpha_{31}}$}
\put(52.5, \hev){$L_{\beta_{3}}$}
\put(58.5, \hev){$E_{\alpha_{32}}$}
\put(48, \he){\circle{1}}
\put(54, \he){\circle{1}}
\put(60, \he){\circle{1}}
\put(48.5, \he){\line(5, 0){5}}
\put(54.5, \he){\line(5, 0){5}}
\put(65.5, \hev){$E_{\alpha_{1}\sprime}$}
\put(71.5, \hev){$E_{\alpha_{2}\sprime}$}
\put(77.5, \hev){$E_{\alpha_{3}\sprime}$}
\put(67, \he){\circle{1}}
\put(73, \he){\circle{1}}
\put(79, \he){\circle{1}}
\put(82, \he){}
\end{picture}
$$
Here $\beta_1, \beta_2, \beta_3$ are the three elements of $B$ contained in $\alpha$,
the three indices
$\alpha, \alpha_{\nu 1}, \alpha_{\nu  2}$ are  the three elements of $A$ containing $\beta_\nu $ for $\nu=1,2,3$,
and $\alpha\sprime_1, \alpha\sprime_2, \alpha\sprime_3$ are  the three elements of $A$ containing $\bar{\alpha}\in B$.
In particular,
the singular locus of the branch curve of $\bar{\pi}_{\alpha}\colon \barX\to\P^2$ consists of $6$ simple singular points,
and its $ADE$-type is $3 A_3+3 A_1$.
Then the eigenspace $V_{1}:=\Ker(g_{\alpha}-I_{16})$ of $g_{\alpha}$ in $S_X\tensor \Q$ 
is of dimension $10$ spanned by the classes
$$
h_Q-[E_{\alpha}],
\quad
[L_{\beta_\nu}],
\quad
[E_{\alpha_{\nu 1}}]+[E_{\alpha_{\nu 2}}],
\quad
[E_{\alpha\sprime_\mu}] \qquad (\nu, \mu=1,2,3),
$$
and the eigenspace $V_{-1}:=\Ker(g_{\alpha}+I_{16})$ 
is the orthogonal complement of $V_{1}$.
Hence we can  calculate the matrix representation of $g_{\alpha}$.
See the webpage~\cite{compdataHesseEnr} for
these matrices.
We can confirm by direct calculation of products of matrices
that $g_{\alpha}$ and $g_{\enrinvol}$ commute.
Therefore
$\enrinvol$ and $\iota_{\alpha}$ commute by Proposition~\ref{prop:period}\,(2).
By Proposition~\ref{prop:DXg},
we know that $\DX ^{g_{\alpha}}$ is an induced chamber.
We can confirm  that $v_{\alpha}^{g_{\alpha}}=-v_{\alpha}$,
and that
$\intf{\hX^{g_{\alpha}}, v}>0$ holds for all primitive defining vectors $v$ of walls of $\DX$
other than $v_{\alpha}$.
Hence
$\DX ^{g_{\alpha}}$ is adjacent to $\DX $
across the inner wall $\DX \cap (v_{\alpha})\sperp$ of $\DX$.
\end{proof}
\par
In~\cite{MR1897389},
Dolgachev and Keum
also constructed automorphisms
of $X$ whose action on $\SX$ maps $\DX$
to the induced chamber adjacent to $\DX$
across each inner wall of type (c) and type (d).
Thus they obtained a set of generators of $\Aut(X)$.
See~\cite{compdataHesseEnr} for the matrix representations
of these automorphisms.
\begin{lemma}\label{lem:sigma}
Let $\sigma_{\alpha}$ and $\sigma_{\beta}$
denote the reflections of $\SX$ with respect to the roots $[E_{\alpha}]$ and $[L_{\beta}]$,
respectively.
Then $\DX^{\sigma_{\alpha}}$ and $\DX^{\sigma_{\beta}}$ are  induced chambers.
\end{lemma}
\begin{proof}
Since $[E_{\alpha}]$ and $[L_{\beta}]$ are roots of $L_{26}$,
the reflections $\sigma_{\alpha}$ and $\sigma_{\beta}$ are the restrictions to $\SX$ of reflections of $L_{26}$.
 \end{proof}
 Combining this fact with the automorphisms of $X$ constructed in~\cite{MR1897389},  
we obtain the following:
\begin{corollary}\label{cor:SXsimpB}
The embedding $i\colon \SX\inj L_{26}$
is of simple Borcherds type.
\end{corollary}
\section{Geometry of the Enriques surface $Y$}\label{sec:main}
\emph{From now on, we denote by $Y$ the quotient of the $K3$ surface $X$
 by the Enriques involution $\enrinvol$ given in Proposition~\ref{prop:enrinvol}}.
%
\subsection{Chamber $\DY$ and generators of $\aut(Y)$}\label{sec:genautY}
\begin{table}
{\small
\begin{eqnarray*}
\baru_{\{1, 2, 3 \} } & =&[0, 0, 1, 0, 1, 2, 1, 1, 0, 0 ], \\
\baru_{\{1, 2, 4 \} } & =&[0, 0, 1, 1, 1, 2, 2, 2, 1, 1 ], \\
\baru_{\{1, 2, 5 \} } & =&[0, 1, -2, -1, -2, -3, -2, -1, -1, -1 ], \\
\baru_{\{1, 3, 4 \} } & =&[0, 0, 1, 0, 0, 0, 0, 0, 0, 0 ], \\
\baru_{\{1, 3, 5 \} } & =&[0, 0, 1, 0, 1, 2, 2, 1, 1, 0 ], \\
\baru_{\{1, 4, 5 \} } & =&[1, 0, -2, -1, -3, -4, -3, -2, -1, -1 ], \\
\baru_{\{2, 3, 4 \} } & =&[0, 0, 1, 1, 2, 2, 1, 1, 1, 1 ], \\
\baru_{\{2, 3, 5 \} } & =&[0, 1, -2, -1, -2, -3, -3, -3, -2, -1 ], \\
\baru_{\{2, 4, 5 \} } & =&[1, 1, -5, -3, -6, -9, -7, -5, -3, -1 ], \\
\baru_{\{3, 4, 5 \} } & =&[1, 0, -2, -1, -2, -4, -3, -3, -2, -1 ], \mystrutd{5pt} \\
\barv_{\{1, 2, 3 \} } & =&[0, 1, -1, 0, -1, -2, -1, -1, 0, 0 ],  \mystruth{10pt} \\
\barv_{\{1, 2, 4 \} } & =&[1, 1, -4, -3, -5, -8, -7, -5, -3, -2 ], \\
\barv_{\{1, 2, 5 \} } & =&[1, 0, -1, -1, -2, -2, -2, -2, -1, 0 ], \\
\barv_{\{1, 3, 4 \} } & =&[1, 0, -1, 0, 0, 0, 0, 0, 0, 0 ], \\
\barv_{\{1, 3, 5 \} } & =&[1, 1, -4, -2, -5, -8, -7, -5, -4, -2 ], \\
\barv_{\{1, 4, 5 \} } & =&[0, 1, -1, -1, -1, -2, -1, -1, -1, 0 ], \\
\barv_{\{2, 3, 4 \} } & =&[1, 1, -4, -3, -6, -8, -6, -5, -4, -2 ], \\
\barv_{\{2, 3, 5 \} } & =&[1, 0, -1, -1, -1, -2, -1, 0, 0, 0 ], \\
\barv_{\{2, 4, 5 \} } & =&[0, 0, 2, 2, 3, 4, 3, 2, 1, 0 ], \\
\barv_{\{3, 4, 5 \} } & =&[0, 1, -1, -1, -2, -2, -2, -1, 0, 0 ]
\end{eqnarray*}
}
\vskip -.2cm
\caption{Primitive defining vectors of the walls of $\DY$}\label{table:wallsDY}
\end{table}
As in Section~\ref{subsec:appEnriques},
we identify the $\Z$-module  $S_Y$ with the $\Z$-submodule   $S_X^{+}$ of $\SX$ by $\pi^*$
so that we have
$$
\PPP_Y=(\SY\tensor\R)\cap \PPP_X
\quad\textrm{and}\quad \nefbig(Y)=\nefbig(X)\cap \PPP_Y.
$$
We have fixed a basis $\eta_1, \dots, \eta_{10}$ of $S_X^+$ in such a way that
the homomorphism~\eqref{eq:homtoetas} is an isometry $\Lten(2)\isom S_X^+$.
We use $\eta_1, \dots, \eta_{10}$ as a basis of $S_Y$,
and write elements of $S_Y\tensor\R$ as row vectors.
Hence the Gram matrix of $\SY$
is equal to the standard Gram matrix of $\Lten=U\oplus E_8$.
\par
Recall that $\SX$ is embedded in $L_{26}$ by the matrix in Table~\ref{table:embSXL}.
Let
$$
\pr_S^+ \dcolon  L_{26}\tensor\R\maprightsp{\pr_S\;}   \SX\tensor \R  \maprightsp{\pr^+}  \SY\tensor \R
$$
denote the composite  of the orthogonal projections
$\pr_S$
and
$\pr^+$.
We consider the tessellation of $\PPP_Y$ given by the locally finite family of hyperplanes~\eqref{eq:prSplushamily}.
We put
$$
\DY:=\PPP_Y\cap \DX.
$$
By Proposition~\ref{prop:autSX}, 
we have $\hX\in \SY$ and hence $\hX\in \DY$.
We put
$$
\hY:=\hX.
$$
Note that 
the class $\hY\in \SY$ is ample on $Y$,  and that $\intfY{\hY, \hY}=10$.
With respect to the basis $\eta_1, \dots, \eta_{10}$ of $\SY$,
we  have
$$
\hY=[ 3, 3, -8, -5, -10, -15, -12, -9, -6, -3 ].
$$
%
%
%
\begin{proposition}
The closed subset $\DY$
of $\PPP_Y$ is an induced chamber
contained in $\nefbig(Y)$.
The set of primitive defining  vectors of walls of $\DY$
consists of $10+10$ vectors
$$
\baru_{\alpha}:=2\,\pr^+([E_{\alpha}])=2\,\pr^+([L_{\bar\alpha}]), \;\;\;\; \barv_{\alpha}:=2\,\pr^+(v_{\alpha})
$$
of $S_Y\dual$,
where $\alpha$ runs through $A$,
and $v_{\alpha}\in \SX\dual$ is the primitive defining vector of an inner wall  of $\DX$ of type {\rm (b)},
which is characterized by~\eqref{eq:valpha}.
The vector representations of $\baru_{\alpha}$ and $\barv_{\alpha}$ are given in
Table~\ref{table:wallsDY}.
\end{proposition}
\begin{proof}
Since $\hX$ is an interior point of $\DX $ in $S_X\tensor \R$,
the vector $\hY=\hX\in \DY$ is an interior point of  $\DY$ in $S_Y\tensor \R$.
Hence $\DY$ is a chamber.
Since $\DX $ is an induced chamber contained in $\nefbig(X)$,
the chamber $\DY$ is an induced chamber contained in $\nefbig(Y)=\PPP_Y\cap \nefbig(X)$.
By definition, the chamber  $\DY$ is defined by the set of vectors
$$
\set{\pr^+(v)}{v\in \FFF(\DX ), \;\;\;\intfY{\pr^+(v), \pr^+(v)}<0},
$$
where $\FFF(\DX )$ is the set of primitive defining vectors of  walls of $\DX $,
which we have calculated in the proof of Proposition~\ref{prop:wallsDX}. 
Using Algorithm~\ref{algo:wall}, we obtain the  set of primitive defining vectors of walls of $\DY$
as Table~\ref{table:wallsDY}.
\end{proof}
%
%
%
\begin{remark}
If $v\in \FFF(\DX)$ defines a wall of $\DX$ of type (c) or (d),
then we have $\intfY{\pr^+(v), \pr^+(v)}=0$,
and hence the hyperplane $(v)\sperp$ of $\PPP_X$
does not intersect $\PPP_Y\subset \PPP_X$.
\end{remark}
\begin{corollary}\label{cor:outerwallsDY}
The vector $\baru_{\alpha}$ is the class of the smooth rational curve $\pi(E_{\alpha})=\pi(L_{\bar\alpha})$
on $Y$.
In particular,
each  of the walls $\DY\cap (\baru_{\alpha})\sperp$ is a wall of $\nefbig(Y)$.
\end{corollary}
The set of primitive defining vectors of walls of $\DY$ is denoted by
$$
\FFF(\DY)\;\;:=\;\;\shortset{\baru_{\alpha}}{\alpha \in A}\;\cup\;\shortset{\barv_{\alpha}}{\alpha\in A}.
$$
We have
\begin{equation}\label{eq:intnumbsuuww}
\intfY{\baru_{\alpha}, \baru_{\alpha\sprime}}=
\begin{cases}
-2 &  \textrm{if $\alpha=\alpha\sprime$,}\\
1 &  \textrm{if $|\alpha\cap\alpha\sprime|=1$,}\\
0 & \textrm{otherwise,}
\end{cases}
\quad
\intfY{\barv_{\alpha}, \barv_{\alpha\sprime}}=
\begin{cases}
-2 &  \textrm{if $\alpha=\alpha\sprime$,}\\
1 &  \textrm{if $|\alpha\cap\alpha\sprime|=2$,}\\
0 & \textrm{otherwise,}
\end{cases}
\end{equation}
and
\begin{equation}\label{eq:intnumbsuw}
\intfY{\baru_{\alpha}, \barv_{\alpha\sprime}}=2\,\delta_{\alpha\alpha\sprime}.
\end{equation}
Moreover, we have
\begin{equation*}\label{eq:intnumbsuwhY}
\intfY{\baru_{\alpha}, \hY}=1,\quad \intfY{\barv_\alpha, \hY}=2.
\end{equation*}
\par
%
%
%
\begin{proposition}\label{prop:DYg}
If $\barg\in \aut(Y)$, then $\DY^{\barg}$ is an induced chamber.
\end{proposition}
\begin{proof}
By Propositions~\ref{prop:Zaut} and~\ref{prop:period},
we have an element $g\in Z_{\aut(X)}(g_{\enrinvol})$ such that
 $g|_{\SY}=\barg$.
Since $\DY^{\barg}=\PPP_Y\cap \DX ^{g}$,
the statement  follows from Proposition~\ref{prop:DXg}.
\end{proof}
\begin{proposition}\label{prop:autDY}
The automorphism group $\aut(\DY )$ of the chamber $\DY$ is
isomorphic to the symmetric group  of degree $5$.
Every element $\barg$ of $\aut(\DY)$ satisfies $\hY^{\barg}=\hY$.
The intersection  $\aut(\DY)\cap \aut(Y)$ is trivial.
\end{proposition}
\begin{proof}
Let $\barg$ be an element of $\aut(\DY)$.
Then $\barg$ induces a permutation of
$\FFF(\DY)$
that preserves~\eqref{eq:intnumbsuuww} and~\eqref{eq:intnumbsuw}.
Note that a vector $v\in \FFF(\DY)$ belongs
to $\shortset{\baru_{\alpha}}{\alpha \in A}$ (resp.~to $\shortset{\barv_{\alpha}}{\alpha\in A}$)
if  there exist exactly three (resp.~six) vectors $v\sprime\in \FFF(\DY)$
such that $\intfY{v, v\sprime}=1$.
Hence the permutation of $\FFF(\DY)$ induced by $\barg$
induces a permutation of $A$
preserving the set of pairs $\{\alpha, \alpha\sprime\}$ satisfying  $|\alpha\cap\alpha\sprime|=1$.
Therefore this permutation comes from a permutation of $\{1, \dots, 5\}$.
Conversely, it can be easily checked that
 a permutation of $\{1, \dots, 5\}$ induces an isometry of $S_Y$.
Hence the first assertion is proved.
From~\eqref{eq:hXsum}, we have
$\hY=\sum_{\alpha\in A} \baru_{\alpha}$.
Hence the second assertion follows.
Suppose that $\barg\in \aut(\DY)$ belongs to $\aut(Y)$.
By Propositions~\ref{prop:Zaut} and~\ref{prop:period},
there exists an element $g\in Z_{\aut(X)}(g_{\enrinvol})$
such that $g|_{\SY}=\barg$.
Since $\hX=\hY$ and $\hY^{\barg}=\hY$,
we have $\hX^g=\hX$.
Since $\DX^g\cap \DX$ contains an interior point $\hX$ of $\DX$,
we see that $\DX^g=\DX$.
By Proposition~\ref{prop:enrinvol},
it follows that $g\in \{\id, g_{\enrinvol}\}$.
Therefore $\barg$ is the identity.
\end{proof}
\par
The involution $\iota_{\alpha}\colon X\to X$
in Proposition~\ref{prop:iotaalpha}
commutes with $\enrinvol$.
Hence $\iota_{\alpha}$ induces an involution $j_{\alpha}\colon Y\to Y$
of $Y$,
whose  representation on $S_Y$ is
$$
\barg_{\alpha}:=g_{\alpha}|_{S_Y},
$$
where $g_{\alpha}\in \aut(X)$ is the representation of $\iota_{\alpha}$
calculated in Proposition~\ref{prop:iotaalpha}.
We calculate the matrix representations of $\barg_{\alpha}$.
(See the webpage~\cite{compdataHesseEnr} for
these matrices.)
\begin{proposition}\label{prop:autYgens}
The element  $\barg_{\alpha} \in \aut(Y)$ maps $\DY$ to the induced chamber
adjacent to $\DY$ across the wall $\DY\cap (\barv_{\alpha})\sperp$.
\end{proposition}
\begin{proof}
By Proposition~\ref{prop:DYg},
we know that $\DY^{\barg_{\alpha}}$ is an induced chamber.
It is easy to check that $v_{\alpha}\in \SX^+\tensor\R$ and $v_{\alpha}^{g_{\alpha}}=-v_{\alpha}$,
and hence
 $\barv_{\alpha}^{\barg_{\alpha}}=-\barv_{\alpha}$.
 Therefore
the vector $-\barv_{\alpha}$ defines a wall of $\DY^{\barg_{\alpha}}$.
Since
$\intf{\hY^{\barg_{\alpha}}, \barv}>0$ holds for all vectors  $\barv$ of  $\FFF(\DY)\setminus\{\barv_{\alpha}\}$,
we see that $\DY^{\barg_{\alpha}}$ is adjacent to $\DY$ across the wall $\DY\cap (\barv_{\alpha})\sperp$.
\end{proof}
%
%
\begin{remark}
Each involution $\barg_{\alpha}$ has an eigenvalue $1$ with multiplicity $6$.
\end{remark}
\begin{proposition}\label{prop:barsigma}
For $\alpha\in A$,
let $\barsigma_{\alpha}\in \OG(\SY)$ denote the reflection of $\SY$ with respect to the root $\baru_{\alpha}$.
Then
$\barsigma_{\alpha}$ maps $\DY$ to the induced chamber
adjacent to $\DY$ across the wall $\DY\cap (\baru_{\alpha})\sperp$.
\end{proposition}
\begin{proof}
First remark that,  since $\SY$ is embedded in $L_{26}$ as $\SY(2)$,
the vector $\baru_{\alpha}$ is not a root of $L_{26}$.
Hence
the argument of Lemma~\ref{lem:sigma}
does not work in this case.
As was shown in~Lemma~\ref{lem:sigma},
the chamber $\DX^{\sigma_{\alpha}\sigma_{\bar\alpha}}$ is an induced chamber in $\PPP_X$.
By direct calculation,
we see that $\sigma_{\alpha}\sigma_{\bar\alpha}$ commutes with $g_{\enrinvol}$
and that the restriction of $\sigma_{\alpha}\sigma_{\bar\alpha}$ to $\SY$ is equal to $\barsigma_{\alpha}$.
Hence $\DY^{\barsigma_{\alpha}}$ is equal to $\PPP_Y\cap\DX^{\sigma_{\alpha}\sigma_{\bar\alpha}}$,
and  $\DY^{\barsigma_{\alpha}}$
is an induced chamber.
Since $\barsigma_{\alpha}$ is the reflection in the hyperplane $(\baru_{\alpha})\sperp$,
it is obvious that $\DY^{\barsigma_{\alpha}}$ is adjacent to $\DY$ across the wall $\DY\cap (\baru_{\alpha})\sperp$.
\end{proof}
From these propositions, we obtain the following corollaries.
\begin{corollary}\label{cor:SYBor}
The primitive embedding $S_Y(2) \inj S_X\inj L_{26}$ of $S_Y(2)$ into $L_{26}$ is of simple Borcherds type.
\end{corollary}
\begin{corollary}\label{cor:autYgen}
The group $\aut(Y)$ is generated by the $10$ involutions $\barg_{\alpha}$.
The induced chamber $\DY$ is a fundamental domain of the action of $\aut(Y)$ on $\nefbig(Y)$.
In particular,
the mapping $\barg\mapsto \DY^{\barg}$ gives rise to a bijection from $\aut(Y)$
to the set of induced chambers contained in $\nefbig(Y)$.
\end{corollary}
\subsection{Smooth rational curves on $Y$}
We have the following:
\begin{proposition}[Lemma 3.1 of~\cite{KondoH}]\label{prop:KondoSXminus}
The lattice $\SX^-$  is isomorphic to $E_6(2)$,
where $E_6$ is the negative-definite root lattice of type $E_6$.
\end{proposition}
Hence the set $\TTT$ of vectors of square-norm $-4$ in  $\SX^-$
consists of $72$ elements.
By the method  in Section~\ref{subsec:ratcurvesY},
we calculate the set $\RRR_d$ of the classes $[C]$ of smooth rational curves
$C$ on $Y$ with $\intfY{[C], \hY}=d$ for $d=1, \dots, \maxSRCdeg$.
%
\begin{proposition}\label{prop:RRRdY}
Let $d$ be a positive integer $\le \maxSRCdeg$.
If  $d\not\equiv 1 \bmod 4$, then $\RRR_d$ is empty.
 If $d\equiv 1 \bmod 4$,
then the cardinality of $\RRR_d$ is as follows.
$$
\arraycolsep=4.5pt
\begin{array}{c|ccccccccccccc}
d & 1 & 5 & 9 & 13 & 17 & 21 & 25 & 29 & 33 &37 &41 &45\\
\hline
|\RRR_d | &  10 & 10 & 60 & 180 & 480 & 750 & 1440 & 2880 & 4110 &5640 &9480 &11280\mystruth{10pt}
\end{array}
$$
\end{proposition}
\subsection{Faces of  $\DY$ and defining relations of $\aut(Y)$}\label{subsec:faces}
For the sake of readability,
we  will use the following notation.
$$
w(\alpha):=\DY\cap (\barv_{\alpha})\sperp,
\quad
\barg(\alpha):=\barg_{\alpha}.
$$
We say that a  face $F$ of $\DY$ is \emph{inner}
if  a general point of $F$ is an interior point of $\nefbig(Y)$,
and that $F$ is \emph{outer} otherwise.
(An ideal face is obviously outer.)
In particular,
the walls $\DY\cap (\baru_{\alpha})\sperp$ are outer,
and $\DY\cap (\barv_{\alpha})\sperp$ are inner.
\par
\begin{table}
$$
\arraycolsep=5pt
\begin{array}{lccccccccc}
\dim &1 & 2 & 3 & 4 & 5 & 6 & 7 & 8 & 9  \mystrutd{5pt}\\
\hline
\textrm{outer faces} &657 & 3420 & 7250 & 8525 & 6270 & 2940 & 840 & 135 & 10 \mystruth{10pt}\\
\textrm{$\aut(Y)$-classes}  &44 & 314 & 1077 & 1759 & 1669 & 1060 & 435 & 105 & 10 \mystrutd{5pt}\\
\hline
\textrm{inner faces} &0 & 0 & 0 & 0 & 0 & 60 & 90 & 45 & 10 \mystruth{10pt}\\
\textrm{$\aut(Y)$-classes}  &0 & 0 & 0 & 0 & 0 & 1 & 15 & 25 & 10
\end{array}
$$
\vskip .1cm
\caption{Numbers of faces of $\DY$ and their $\aut(Y)$-equivalence classes}\label{table:numbfaces}
\end{table}
\begin{definition}\label{def:autYclass}
Let $F$ and $F\sprime$ be faces of $\DY$.
We put $F\sim F\sprime$ if there exists an element $\barg\in \aut(Y)$ such that
$F\sprime=F^{\barg}$;
that is, the induced chambers $\DY$ and $\DY^{\barg}$ share the face $F\sprime$
and the face $F$ of $\DY$ is mapped to the face $F\sprime$ of $\DY^{\barg}$ by $\barg$.
It is obvious that $\mathord{\sim}$ is an equivalence relation.
When $F\sim F\sprime$,
we say that $F$ and $F\sprime$ are \emph{$\aut(Y)$-equivalent}.
\end{definition}
\begin{proposition}\label{prop:faces}
The numbers of faces of $\DY$ and  their $\aut(Y)$-equivalence classes are given in Table~\ref{table:numbfaces}.
\end{proposition}
\begin{proof}
The set of faces can be calculated from the set of walls of $\DY$
by using Algorithm~\ref{algo:wall} iteratively.
A face $F$ is outer if and only if
there exists an outer wall of $\DY$ containing $F$.
Therefore we can make the lists of all outer faces
and of all inner faces.
\par
The set of $\aut(Y)$-equivalence classes of faces
is calculated by the following method.
We put $F\simw F\sprime$
if there exists an inner  wall $w(\alpha)$ of $\DY$ containing the face $F$
such that $F=F\sp{\prime \barg(\alpha)}$.
(The superscript $a$ in the symbol \,$\simwvoid$\,  is intended to mean ``adjacent".)
We show that
the $\aut(Y)$-equivalence relation
 $\mathord{\sim}$ is the smallest equivalence relation containing
 the  relation $\simwvoid$.
Indeed, it is obvious that $F\simw F\sprime$ implies $F\sim F\sprime$.
Suppose that $F\sim F\sprime$,
and let $\barg\in \aut(Y)$ be an element such that $F=F\sp{\prime \barg}$.
Looking at the tessellation of $\nefbig(Y)$ by induced chambers
locally around a general point of $F$,
we see that there exists a sequence of induced chambers
$$
D_0=\DY, \; D_1, \; \dots, \; D_m=\DY^{\barg}
$$
with the following properties:
\begin{enumerate}[(a)]
\item Each $D_i$ contains $F$ and is contained in $\nefbig(Y)$.
\item For $i=1, \dots, m$,
the induced chambers $D_{i-1}$ and $D_{i}$ are adjacent across a wall containing $F$.
\end{enumerate}
Let $\barg_i$ be an element of $\aut(Y)$ such that $D_i=\DY^{\barg_i}$.
Note that $\barg_i$ is unique by Corollary~\ref{cor:autYgen}.
Let $w_i$ be the wall between $D_{i-1}$ and $D_i$.
Since both of $D_{i-1}$ and $D_i$ are contained in $\nefbig(X)$,
there exists an inner wall
$w(\alpha_i)$ of $\DY$
that is mapped to $w_i$ by $\barg_{i-1}$.
Then we have $\barg_i=\barg(\alpha_i)\barg_{i-1}$, and hence
$$
\barg_i=\barg(\alpha_i)\cdots \barg(\alpha_1).
$$
Let $F_i$ be the face of $\DY$ that is mapped to the face $F$ of $D_i$
by $\barg_i$.
Since $\barg=\barg_m$, we have $F_m=F\sprime$.
Since
$F_{i-1}=F_i^{\barg(\alpha_i)}$,
we have $F_{i-1}\simw F_i$.
Therefore $F_0=F$ and $F_m=F\sprime$
are equivalent under the minimal equivalence relation containing $\mathord{\simw}$.
\par
For each face $F$ of $\DY$,
we can make the finite list of all faces $F\sprime$ of $\DY$ such that $F\simw F\sprime$.
From these lists,  we calculate the set of  $\aut(Y)$-equivalence classes of faces.
\end{proof}
We give a description of inner faces of $\DY$ of codimension $2$.
Let $w(\alpha):=\DY\cap (\barv_{\alpha})\sperp$ be an inner wall.
For any $\alpha\sprime\in A\setminus \{\alpha\}$,
the space $F_{\alpha\sprime}:=w(\alpha)\cap  (\barv_{\alpha\sprime})\sperp$
contains a non-empty open subset
of $(\barv_{\alpha})\sperp\cap (\barv_{\alpha\sprime})\sperp$.
Indeed, the image $\pr(\hY)$ of $\hY$ by the orthogonal projection to $(\barv_{\alpha})\sperp\cap (\barv_{\alpha\sprime})\sperp$
satisfies $\intfY{\pr(\hY), \baru_{\alpha\spprime}}>0$ for all $\alpha\spprime\in A$ and
$\intfY{\pr(\hY), \barv_{\alpha\spprime}}>0$ for all $\alpha\spprime\in A\setminus\{\alpha, \alpha\sprime\}$.
Therefore the inner wall  $w(\alpha)$ contains exactly  $9$ inner faces $F_{\alpha\sprime}$ of codimension $2$.
Let $x$ be a general point of $F_{\alpha\sprime}$.
If $\intfY{\barv_{\alpha}, \barv_{\alpha\sprime}}=0$, then $(\barv_{\alpha})\sperp$ and $ (\barv_{\alpha\sprime})\sperp$ intersect perpendicularly at $x$
and hence $x$ is contained in exactly $4$ induced chambers of $\nefbig(Y)$,
while if $\intfY{\barv_{\alpha}, \barv_{\alpha\sprime}}=1$, then
 $(\barv_{\alpha})\sperp$ and $ (\barv_{\alpha\sprime})\sperp$ intersect with angle $\pi/3$ at $x$
 and hence
 $x$ is contained in exactly $6$ induced chambers.
 These induced chambers lead to the relations among $\barg(\alpha)$ in Proposition~\ref{prop:rels} below.
\begin{proposition}\label{prop:aroundF}
Let $F$ be a non-ideal face of $\DY$.
Then the set
$$
\GGG(F):=\set{\barg\in \aut(Y)}{F\subset \DY^{\barg}}
$$
is finite, and can be calculated  explicitly.
\end{proposition}
\begin{proof}
Note that the family of hyperplanes~\eqref{eq:prSplushamily}
that gives the tessellation of $\PPP_Y$ by  induced chambers is locally finite in $\PPP_Y$.
Since $F$ is not an ideal face,
the number of induced chambers containing $F$ is finite.
Hence $\GGG(F)$ is finite.
\par
The set $\GGG(F)$ can be calculated as follows.
We set
$\algoG:=\{\algoid\}$,
where $\algoid$ is the identity of $\aut(Y)$.
Let $\algof$ be the procedure
that takes an element $\barg$ of $\GGG(F)$ as an input
and carries out the following task:
\begin{enumerate}[(a)]
\item
Let $F\sprime$ be the face of $\DY$ such that $F^{\prime\barg}$ is
equal to $F$.
We calculates the set
$\{w(\alpha_1), \dots, w(\alpha_k)\}$
of inner walls of $\DY$ that contain $F\sprime$.
\item
For each $j=1, \dots, k$,
we calculate $\barg\sprime:=\barg(\alpha_j)\barg$,
which is an element of $\GGG(F)$,
and if $\barg\sprime$ is not yet in the set $\algoG$,
we add $\barg\sprime$ in $\algoG$
and input $\barg\sprime$ to  the procedure $\algof$.
\end{enumerate}
We input $\algoid$ to  the procedure $\algof$.
It is easy to see that,
when the whole procedure terminates,
the set $\algoG$ is equal to
the set $\GGG(F)$.
\end{proof}
\begin{proposition}\label{prop:rels}
The following relations form a set of defining relations of $\aut(Y)$ with respect to
the set of generators
$\shortset{\barg(\alpha)}{ \alpha\in A}$;
$$
\barg(\alpha)^2=\id
$$
for each $ \alpha\in A$,
$$
(\,\barg(\alpha)\, \barg(\alpha\sprime)\, \barg(\alpha\spprime)\,)^2=\id
$$
for each triple $(\alpha, \alpha\sprime, \alpha\spprime)$
of distinct elements of $A$
such that $|\alpha\cap \alpha\sprime\cap\alpha\spprime|=2$,
and
$$
(\,\barg(\alpha) \, \barg(\alpha\sprime)\,)^2=\id
$$
for each pair $(\alpha, \alpha\sprime)$
such that $|\alpha\cap \alpha\sprime|=1$.
\end{proposition}
\begin{proof}
By the standard argument of the geometric group theory,
there exists a one-to-one correspondence
between a set of
defining relations except for $\barg(\alpha)^2=\id$ and
the set of $8$-dimensional inner faces of $\DY$.
Let $F$ be an $8$-dimensional inner face of $\DY$.
Then there exist exactly two walls of $\DY$ containing $F$,
and they are both inner.
We put $D_0:=\DY$,
and choose an induced chamber $D_1$ from the two induced chambers
that contain $F$ and are adjacent to $D_0$.
Then there exists a cyclic sequence
$$
D_0, D_1, \dots, D_{m-1}, D_m=D_0
$$
of induced chambers in $\nefbig(Y)$
with the following properties:
\begin{enumerate}[(a)]
\item For each $i\in \Z/m\Z$,
 $D_{i-1}$ and $D_{i+1}$
are the two induced chambers that contain $F$ and are adjacent to $D_{i}$.
\item If $i, j\in \Z/m\Z$ are distinct, then $D_i$ and $D_j$ are distinct.
\end{enumerate}
We calculate  the  sequence of inner walls $w(\alpha_1), \dots, w(\alpha_m)$
of $\DY$ such that
$$
D_i=\DY^{\barg(\alpha_i) \cdots \barg(\alpha_1)}.
$$
Then we have $\barg(\alpha_m) \cdots \barg(\alpha_1)=\id$,
and this is the defining relation corresponding to the inner face $F$.
The cycle $D_0, D_1, \dots, D_m$ for each $F$
can be computed by
Proposition~\ref{prop:aroundF}.
Thus
we obtain a list of defining relations
from  the list of inner faces  of $\DY$ that
we have calculated in Proposition~\ref{prop:faces}.
\par
The $45=10\times 3 +15$ inner faces of dimension $8$
are decomposed into $25=10+15$ $\aut(Y)$-equivalence classes.
We see that,
if $w(\alpha)$ and $w(\alpha\sprime)$ are distinct inner walls, then
$w(\alpha)\cap w(\alpha\sprime)$ is an $8$-dimensional  inner face.
For each $\beta\in B$,
there exist exactly three elements $\alpha, \alpha\sprime, \alpha\spprime\in A$
that contain $\beta$,
and the three $8$-dimensional  inner faces
$w(\alpha)\cap w(\alpha\sprime)$,
$w(\alpha\sprime)\cap w(\alpha\spprime)$
and
$w(\alpha\spprime)\cap w(\alpha)$
form an $\aut(Y)$-equivalence class.
The face $w(\alpha)\cap w(\alpha\sprime)$
corresponds
to the relation
$$
\barg(\alpha)\, \barg(\alpha\spprime)\, \barg(\alpha\sprime)\,\barg(\alpha)\, \barg(\alpha\spprime)\, \barg(\alpha\sprime)\;=\;\id.
$$
For each $i\in \{1, \dots, 5\}$,
there exist exactly three pairs
$\{\alpha_\nu, \alpha\sprime_\nu\}$
($\nu=1,2,3$)
of elements of $A$
such that $\alpha_\nu\cap  \alpha\sprime_\nu=\{i\}$.
For each pair $\{\alpha_\nu, \alpha\sprime_\nu\}$,
the $8$-dimensional  inner face
$w(\alpha_\nu)\cap w(\alpha\sprime_\nu)$
forms an $\aut(Y)$-equivalence class
consisting of only one element.
The face $w(\alpha_\nu)\cap w(\alpha\sprime_\nu)$  corresponds
to the relation
$$
\barg(\alpha_\nu) \, \barg(\alpha\sprime_\nu)\,\barg(\alpha_\nu) \, \barg(\alpha\sprime_\nu)\;=\;\id.
$$
Thus Proposition~\ref{prop:rels} is proved.
\end{proof}
\subsection{Proof of Theorem~\ref{thm:genAut}} \label{subsec:proofgenAut}
In~\cite{MR759266},  Enriques surfaces $Z$ with  automorphisms
that act on $\SZ$ trivially are classified.
(See also~\cite{KondoFinite}  and~\cite{MR2740697}.)
It follows that
the action of  $\Aut(Y)$  on  $\SY$  is faithful.
Then
Theorem~\ref{thm:genAut} follows immediately from
Corollary~\ref{cor:autYgen} and Proposition~\ref{prop:rels}.
\subsection{Elliptic fibrations of $Y$}\label{subsec:ellfibs}
We prove Theorem~\ref{thm:ellfibs}.
The $657=57+600$ one-dimensional faces of $\DY$ are
divided into $44=21+23$ $\aut(Y)$-equivalence classes.
Among them, there exist exactly $57$ ideal faces,
and they are divided into  $21$ $\aut(Y)$-equivalence classes.
\par
Let $\phi\colon Y\to\P^1$
be an elliptic fibration,
and let $2 f_{\phi}\in \SY$ denote
the class of a fiber of $\phi$.
There exists an isometry $g\in \aut(Y)$
that maps $f_{\phi}$ in an ideal face of $\DY$.
Conversely,
let $F$ be an ideal face of $\DY$,
and let $f\in \SY$ be the primitive vector
such that $F=\R_{\ge 0} f$.
Since $f$ is nef and satisfies $\intfY{f, f}=0$,
there exists an elliptic fibration $\phi\colon Y\to\P^1$
such that $f=f_{\phi}$.
Therefore there exists a bijection between the set of
elliptic fibrations modulo the action of $\Aut(Y)$ and
the set of $\aut(Y)$-equivalence classes
of ideal faces of $\DY$.
\par
For each ideal face $F=\R_{\ge 0} f$
with $f\in \SY$ primitive,
the $ADE$-type of reducible fibers
of the corresponding elliptic fibration
can be calculated
from $f$, the ample class $\hY$,
and the sets $\RRR_d$ calculated in Proposition~\ref{prop:RRRdY}
by the method described in
Section~\ref{subsec:ADEellfibY}.
Thus we obtain Table~\ref{table:ellfibs} and hence Theorem~\ref{thm:ellfibs} is proved.
%
\subsection{$\RDP$-configurations  on $Y$}\label{subsec:RDP}
We prove Theorem~\ref{thm:RDP}.
Let $\psi\colon Y\to\barY$ be a birational morphism to
a surface $\barY$ that has only rational double points as its singularities,
and let $h_{\psi} \in \SY$ be the pullback of the class of a hyperplane section  of $\barY$.
Composing $\psi$ with an automorphism of $Y$,
we assume that $h_{\psi}\in \DY$.
We see that the set
$$
\RRR(\psi):=\set{[C]\in \SY}{\hbox{$C$ is a smooth rational curve on $Y$ contracted by $\psi$}}
$$
can be calculated from the face $F$ of $\DY$
that contains $h_{\psi}$ in its interior.
Indeed, since $\intfY{h_{\psi}, h_{\psi}}>0$,
the face $F$ is not an ideal face, and hence
we can calculate the set
$\GGG(F)$ defined in Proposition~\ref{prop:aroundF}.
Then $\RRR(\psi)$ is equal to
$$
\sethd{\baru_{\alpha}^{\barg}}{7cm}{
$\barg$ is an element of $\GGG(F)$
such that the wall $(\DY\cap (\baru_{\alpha})\sperp)^{\barg}$ of $\DY^{\barg}$
contains $F$.}
$$
Therefore we write $\RRR(\psi)$ as $\RRR(F)$.
Conversely,
let $F$ be a non-ideal face of $\DY$,
and let $h_F$ be an element of $F\cap \SY$
that is not contained in any wall of $F$.
Multiplying $h_F$ by a positive integer if necessary,
we can assume that the line bundle $L_F\to Y$
corresponding to $h_F$ is globally generated
and defines a morphism $\Phi_{|L_F|}\colon Y\to \P^m$.
Let
$$
Y\maprightsp{\psi} \barY\maprightsp{} \P^m
$$
be the Stein factorization  of $\Phi_{|L_F|}$.
Then we have $\RRR(\psi)=\RRR(F)$.
\par
We calculate $\RRR(F)$
for all non-ideal faces $F$ of $\DY$.
For two non-ideal faces $F$ and $F\sprime$,
we put $F\le F\sprime$ if $F$ is a face of $F\sprime$
and $\RRR(F)=\RRR(F\sprime)$ holds.
In the previous version of this paper,
we look at the maximal faces with respect to this partial ordering,
and divide them into $\aut(Y)$-equivalence classes.
However, 
even when two maximal faces $F$ and $F\sprime$
are \emph{not} $\aut(Y)$-equivalent,
the $\RDP$-configurations $\RRR(F)$ and $\RRR(F\sprime)$
can be in the same $\Aut(Y)$-orbit.
This happens, for example,
when $F^g$ and $F\sprime$ span a same linear subspace
in $S_Y\tensor \R$ for some automorphism $g\in \aut(Y)$.
The algorithm to calculate representatives
of this new equivalence relation
is explained in our new paper~\cite{ShimadaAutEnr}
in a more general setting.
\erase{
For two non-ideal faces $F$ and $F\sprime$,
we put $F\le F\sprime$ if $F$ is a face of $F\sprime$
and $\RRR(F)=\RRR(F\sprime)$ holds.
Looking at the maximal faces with respect to this partial ordering,
and dividing them into $\aut(Y)$-equivalence classes,
we obtain Table~\ref{table:RDP}
and hence Theorem~\ref{thm:RDP} is proved.
For the algorithm of the calculation
of $\aut(Y)$-equivalence classes, see~\cite{ShimadaAutEnr}.}
%
%
\subsection{Vinberg chambers in $\DY$}\label{subsec:generic}
In this subsection,
we identify $\SY$ with $\Lten$ by the isometry $\Lten \isom \SY$
given by~\eqref{eq:homtoetas}.
In particular,
the chamber $\DY$ is contained  in $\PPP_{10}$.
Note that
the primitive defining vectors $\baru_{\alpha}$ and $\barv_{\alpha}$
of walls of $\DY$ are  roots (see~\eqref{eq:intnumbsuuww}),
and hence $\DY$ is tessellated by Vinberg chambers.
\begin{proof}[Proof of Theorem~\ref{thm:906608640}.]
%
%
%
Let $V_0$ denote the Vinberg chamber $D_{10}$ defined in  Section~\ref{subsec:L10}.
We put $\vare_i(V_0):=e_i$ for $i=1, \dots, 10$,
and let $m_i(V_0)$ denote the wall $V_0\cap (\vare_i(V_0))\sperp$ of $V_0$.
Then,
for each $i=1, \dots, 9$,
there exists a unique non-ideal one-dimensional face $F_i(V_0)$ of $V_0$
that is not contained in the wall $m_i(V_0)$.
We denote by $f_i(V_0)\in \Lten$ the primitive vector
such that $F_i(V_0)=\R_{\ge 0}\,f_i(V_0)$.
Then we have
$f_1(V_0)=\hY$
under the identification $\Lten=\SY$.
(In fact,
we have chosen the isomorphism~\eqref{eq:homtoetas}
in such a way that $f_1(V_0)=\hY$ holds.)
\par
Let $V$ be an arbitrary Vinberg chamber in $\PPP_{10}$.
Since the automorphism group $\aut(V_0)$
of $V_0$
is trivial,
there exists a unique isometry $\barg(V)\in \OG^+(\Lten)$ that maps $V_0$ to $V$.
We put
$$
\vare_i(V):=\vare_i(V_0)^{\barg(V)},
\;\;
m_i(V):=m_i(V_0)^{\barg(V)},
\;\;
f_i(V):=f_i(V_0)^{\barg(V)}.
$$
We say that a primitive vector $v$ of $\Lten$ is an \emph{$f_1$-vector}
if $v=f_1(V)$ for some Vinberg chamber $V$.
Let $v=f_1(V)$ be an $f_1$-vector.
We put
$$
S(v):=\set{V\sprime}{\textrm{$V\sprime$ is a Vinberg chamber such that $v=f_1(V\sprime)$}}.
$$
Since
the  defining roots $\vare_2(V), \dots, \vare_{10}(V)$ of
the walls $m_2(V), \dots, m_{10}(V)$ of $V$
containing $f_1(V)$ form a Dynkin diagram of type $A_9$,
the cardinality of $S(v)$
is equal to
$|\SSSS_{10}|$.
We then put
$$
\Sigma(v):=\bigcup_{V\sprime \in S(v)} V\sprime,
$$
and call it a \emph{$\Sigma$-chamber
with the center $v$}.
It is obvious that $\PPP_{10}$ is tessellated by $\Sigma$-chambers.
The  defining roots $\vare_2(V), \vare_3(V)$,  $\vare_5(V), \dots, \vare_{10}(V)$ of
the walls  of $V$
that contain $f_1(V)$
and are perpendicular to the wall $m_1(V)$ opposite to $f_1(V)$
form a Dynkin diagram of type $A_2+A_6$.
Hence there exist exactly
$|\SSSS_{3}\times \SSSS_{7}|$
Vinberg chambers $V\sprime$ in $S(v)$
such that $m_1(V)$ and $m_1(V\sprime)$ are supported on the same hyperplane.
Hence
the number of walls of
the chamber $\Sigma(v)$ is
$$
\frac{|\SSSS_{10}|}{|\SSSS_{3}\times \SSSS_{7}|}=\frac{10!}{3!\times 7!}=120.
$$
In particular, we see that
the number of  $\Sigma$-chambers that are adjacent
to the $\Sigma$-chamber $\Sigma(v)$ is $120$.
Moreover,
 we can calculate the list $\{v_1, \dots, v_{120}\}$
 of centers of these adjacent $\Sigma$-chambers.
 \par
 Let $v$ be an $f_1$-vector.
 If $v$ belongs to the interior of $\DY$,
 then all $10!$ Vinberg chambers contained  in $\Sigma(v)$
 are contained in $\DY$.
If $v$ does not belong to  $\DY$,
then none of Vinberg chambers in $\Sigma(v)$
are contained in $\DY$.
Suppose that $v$ is located on the boundary of $\DY$.
We calculate  the Dynkin diagram $\varDelta$
formed by the roots
$$
\shortset{\baru_{\alpha}}{\intfY{\baru_{\alpha}, v}=0}\;\;\cup\;\;\shortset{\barv_{\alpha}}{\intfY{\barv_{\alpha}, v}=0}
$$
that define walls of $\DY$ containing $v$.
This Dynkin diagram  $\varDelta$ is  a
sub-diagram of the Dynkin diagram of type $A_{9}$.
Let $W(\varDelta)\subset \SSSS_{10}$ denote the corresponding
subgroup.
Then, among the  $10!$ Vinberg chambers in $\Sigma(v)$,
exactly $|\SSSS_{10}|/|W(\varDelta)|$ Vinberg chambers
are contained in $\DY$.
\par
Starting from the $f_1$-vector $\hY$,
we cover $\DY$ by $\Sigma$-chambers,
and count the number of Vinberg chambers in $\DY$.
By this method,  Theorem~\ref{thm:906608640} is proved.
\end{proof}
\par
\medskip
Recall that $\Zgen$ is a generic Enriques surface.
The  natural representation $\Aut(\Zgen)\to \OG^+(\Lten)$ is injective by~\cite{MR718937}.
Let $\aut(\Zgen)$ denote the image of this homomorphism.
Let $\wtautY$ denote the subgroup of $\OG^+(\Lten)$
generated by $\aut(Y)$ and the ten reflections $\barsigma_{\alpha}$
with respect to the roots $\baru_{\alpha}\in \SY=\Lten$.
\begin{theorem}\label{thm:compareorder}
{\rm (1)} The group $\wtautY$ contains $\aut(\Zgen)$
as a normal subgroup,
and  $\wtautY/\aut(\Zgen)$ is isomorphic to the Weyl group  $W(E_6)$
of type $E_6$.
{\rm (2)} The induced chamber
$\DY$ is a fundamental domain of the action of $\wtautY$ on $\PPP_Y$.
\end{theorem}
\begin{proof}
Recall that the natural  homomorphism $\rho\colon \OG^+(\Lten)\to \OG(q_{\Lten(2)})\cong \GO^+_{10}(2)$ is surjective.
By~\cite{MR718937},
we know that $\aut(\Zgen)$ is equal to $\Ker \rho$, and hence
$$
[\,\OG^+(\Lten): \aut(\Zgen)\,]=46998591897600=51840\cdot 906608640.
$$
By the brute force method using~\cite{gap},
we see that
the order of the subgroup of $ \GO^+_{10}(2)$
generated by $10+10$ elements $\rho(\barsigma_{\alpha})$ and $\rho(\barg_{\alpha})$ is $51840$.
Hence we have
$$
[\,\wtautY: \wtautY\cap \aut(\Zgen)\,]=51840.
$$
On the other hand,
by Propositions~\ref{prop:DYg},~\ref{prop:barsigma}, and Theorem~\ref{thm:906608640},
we have
$$
[\,\OG^+(\Lten): \wtautY\,]=\frac{906608640}{|\,\wtautY\cap \aut(\DY)\,|}.
$$
Hence we obtain $|\,\wtautY\cap \aut(\DY)\,|=1$,
which implies the assertion (2).
Moreover, we have $\aut(\Zgen)\lhd \wtautY$ and $|\wtautY/\aut(\Zgen)|=51840$.
\par
In the following,
we denote by $D(L)$ the discriminant group $L\dual/L$ of an even lattice $L$,
by $q(L)$ the discriminant form of $L$,
and by $\eta(L)\colon \OG(L)\to \OG(q(L))$
the natural homomorphism.
We have $|D(\SX)|=2^4\cdot 3$, $|D(\SX^+)|=2^{10}$, and by Proposition~\ref{prop:KondoSXminus},
$|D(\SX^-)|=2^6\cdot 3$.
By~\cite{MR525944}, the even overlattice $\SX$ of $\SX^+\oplus \SX^-$
defines an isotropic subgroup
$$
\HS:=\SX/(\SX^+\oplus\SX^-)\subset D(\SX^+)\oplus D(\SX^-)
$$
such that $\HS\cap D(\SX^+)=0$, $\HS\cap D(\SX^-)=0$, and $|\HS\sperp/\HS|=2^4\cdot 3$.
Hence we have $|\HS|=2^6$.
Let $(D(\SX^-)_2, q(\SX^-)_2)$ denote the $2$-part of
$$
(D(\SX^-), q(\SX^-))\cong (D(E_6(2)), q(E_6(2))),
$$
which can be  regarded as a quadratic form of Witt defect $1$
on $\F_2^6$.
The automorphism group $\GO^-_6(2)$ of this quadratic form  is isomorphic to $W(E_6)$.
(See page 26 of~\cite{ATLAS}.)
Since $|\HS|=|D(\SX^-)_2|=2^6$, the second projection induces an isomorphism
$\gamma_H\colon \HS\isom D(\SX^-)_2$,
and the  composite of $\gamma_H\inv$ and the first projection defines an embedding
$$
\gamma \dcolon  (D(E_6(2)), q(E_6(2)))\cong  (D(\SX^-)_2, q(\SX^-)_2) \;\inj\; (D(\SX^+), q(\SX^+)).
$$
Note that the image $H$ of $\gamma$
is equal to  the image of the natural homomorphism $\SX\to \SX^{+\vee} \to D(\SX^+)$.
Let $q_H$ denote the restriction of $q(\SX^+)$ to $H$,
and  $q_{H\sperp}$  the restriction  to the orthogonal complement $H\sperp$  of $H$ in $(D(\SX^+), q(\SX^+))$.
Since $q_H$ is isomorphic to $-q(\SX^-)_2$,
we have $\OG(q_H)\cong {\rm GO}^-_6(2) \cong W(E_6)$.
We consider the homomorphism
$$
\rho\sprime \dcolon \wtautY\maprightinj \OG^+(\SY)\cong \OG^+(\SX^+)\maprightsp{\eta(\SX^+)}\OG(q(\SX^+)).
$$
Since the homomorphism $\OG^+(\SY)\cong \OG^+(\SX^+)\to \OG(q(\SX^+))$ is identified with $\rho$,
the homomorphism  $\rho\sprime$ embeds $\wtautY/\aut(\Zgen)$ into $\OG(q(\SX^+))$.
Every element of the image of
$$
 \wtautY\inj \OG^+(\SY)\cong \OG^+(\SX^+)
 $$
lifts to an element of $ \OG(\SX)$.
Indeed,  $\barg_{\alpha}\in  \wtautY$ lifts to $g_{\alpha}\in \OG(\SX)$,
and $\barsigma_{\alpha}\in  \wtautY$ lifts to $\sigma_{\alpha}\sigma_{\bar{\alpha}}\in \OG(\SX)$ (see the proof of Proposition~\ref{prop:barsigma}).
Hence every element in the image of $\rho\sprime$ preserves the factors $H$ and $H\sperp$
of $D(\SX^+)=H\oplus H\sperp$.
By direct computation, we see that  $\rho\sprime(\barg_{\alpha})$ and $\rho\sprime(\barsigma_{\alpha})$
act on $H\sperp$ trivially.
Therefore $\wtautY/\aut(\Zgen)$ can be regarded as a subgroup of $\OG(q_H)$.
Comparing the order, we obtain $\wtautY/\aut(\Zgen)\cong \OG(q_H)\cong W(E_6)$.
\end{proof}
\begin{remark}\label{rem:specialization}
Note that the lift $\sigma_{\alpha}\sigma_{\bar{\alpha}}\in \OG^+(\SX)$
of $\barsigma_{\alpha}\in  \wtautY$ satisfies the period condition
$\eta_{\SX}(\sigma_{\alpha}\sigma_{\bar{\alpha}})\in \{\pm 1\}$
(see~Proposition~\ref{prop:period}).
By the specialization of  $\Zgen$ to $Y$,
the period condition is weakened and $\aut(\Zgen) $
becomes the larger group $\wtautY$
with $10+10$ generators $\barg_{\alpha}$ and $\barsigma_{\alpha}$.
The presence of smooth rational curves on $Y$, however,
 prevents
the $10$ generators  $\barsigma_{\alpha}$
from entering into $\aut(Y)$.
\end{remark}
%
%
\section{Entropy}\label{subsec:entropy}
Recently,
many works have appeared
on the distribution of entropies $\log \lambda(g)$
of automorphisms $g$ of $K3$  or Enriques surfaces,
where $\lambda(g)$ is the spectral radius
of the action of $g$ on the N\'eron-Severi lattice of the surface.
In particular,
the problem to determine the minimum
of the positive entropies in a certain class of automorphisms
has been studied, for example, in~\cite{MR3437870},~\cite{MR2761934} and~\cite{DolgachevArxiv2016} .
In this section, we report the result of a computational experiment on the entropies of automorphisms of $\Zgen$.
\par
\medskip
By the result of~\cite{MR718937} and   Theorem~\ref{thm:compareorder} above,
we have the following equalities:
\begin{eqnarray*}
\aut(\Zgen)&=&\Ker (\rho\colon \OG^+(\Lten)\surj \GO^+_{10}(2))\\
&=&\Ker (\rho |_{\wtautY } \colon \wtautY \surj  \GO^-_{6}(2)).
\end{eqnarray*}
Since we know finite sets of generators  for $ \OG^+(\Lten)$ and for $\wtautY$,
we can obtain a finite set of generators  of $\aut(\Zgen)$ by the Reidemeister-Schreier method
(see Chapter 2 of~\cite{MKSBook}) from each of these descriptions of $\aut(\Zgen)$.
Since $|\GO^+_{10}(2)|$ is very large, however,
making use of the first equality is not practical.
On the other hand,
since $| \GO^-_{6}(2)|$ is much smaller compared with $|\GO^+_{10}(2)|$,
we have managed to obtain
a finite set of generators of $\aut(\Zgen)$ in a reasonable computation time by means of the second  equality.
 \par
 Using this generating set,
 we search for elements $g\in \aut(\Zgen)$
 with small $\lambda(g)$
 for each degree $d=2,4, \dots, 10$
 of the minimal polynomial $s_{\lambda(g)}$ of the Salem number $\lambda(g)$.
 Below is the list of the smallest values of $\lambda(g)$
among the ones we found by an extensive random search of elements of $\aut(\Zgen)$.
 See~\cite{compdataHesseEnr} for the matrices $g$ with these spectral radii.
 $$
 \begin{array}{clc}
 d & s_{\lambda(g)} \mystrutd{6pt}& \lambda(g) \\
 \hline
 2 &{t}^{2}-14\,t+\dots&13.9282 \mystruth{10pt}\dots\\
4 &{t}^{4}-16\,{t}^{3}+14\,{t}^{2}-\dots&15.1450\dots\\
6 &{t}^{6}-38\,{t}^{5}-49\,{t}^{4}-84\,{t}^{3}-\dots&39.3019\dots\\
8 &{t}^{8}-68\,{t}^{7}+68\,{t}^{6}-188\,{t}^{5}+118\,{t}^{4}-\dots&67.0269\dots\\
10 &{t}^{10}-138\,{t}^{9}-19\,{t}^{8}-248\,{t}^{7}+18\,{t}^{6}-252\,{t}^{5}+\dots&138.1505\dots
 \end{array}
 $$
\begin{remark}
The famous Lehmer's number $\lambda_{\rm Leh}=1.17628\dots$
is the spectral radius of a Coxeter element $c$ of $\OG^+(\Lten)=W(\Lten)$.
The order of $\rho(c)\in \GO^+_{10}(2)$ is $31$,
and the spectral radius of $c^{31}\in\aut(\Zgen)$ is equal to
$\lambda_{\rm Leh}^{31}=153.4056\dots$.
\end{remark}
\par
\medskip
{\bf Acknowledgements.}
Thanks are due to Professors Igor Dolgachev,
Shigeyuki Kondo,
Shigeru Mukai, Viacheslav  Nikulin,
Keiji Oguiso,  Hisanori Ohashi, and Matthias Sch\"utt   for many discussions.
In particular, Professor Kondo suggested the isomorphism $\wtautY/\aut(\Zgen)\cong W(E_6)$ in
Theorem~\ref{thm:compareorder}.
%
The author also thanks the referee for many valuable comments on the first version of this paper.
\bibliographystyle{plain}
\def\cftil#1{\ifmmode\setbox7\hbox{$\accent"5E#1$}\else
  \setbox7\hbox{\accent"5E#1}\penalty 10000\relax\fi\raise 1\ht7
  \hbox{\lower1.15ex\hbox to 1\wd7{\hss\accent"7E\hss}}\penalty 10000
  \hskip-1\wd7\penalty 10000\box7} \def\cprime{$'$} \def\cprime{$'$}
  \def\cprime{$'$} \def\cprime{$'$}

\end{document}